\documentclass[preprint, times]{elsarticle}
\usepackage{amsmath}
\usepackage{verbatim}
\usepackage{waveletnotation}
\usepackage{algorithm,algorithmic}
\usepackage{epsfig}

\def\eop{\hfill\rule{2.0mm}{2.0mm}}
\newcommand{\mrow}{r}
\newcommand{\mcol}{s}

\newtheorem{lemma}{Lemma}

\newtheorem{theorem}{Theorem}
\newtheorem{example}{Example}
\newproof{proof}{Proof}
\begin{document}

\title{Matrix Extension with Symmetry and  Construction of Biorthogonal Multiwavelets}

\author[rvt]{Xiaosheng Zhuang\corref{cor1}}
\ead{xzhuang@math.ualberta.ca}
\address[rvt]{Department of Mathematical and Statistical Sciences,
University of Alberta, Edmonton, Alberta, Canada T6G 2G1. }

\cortext[cor1]{Corresponding author, {\tt
http://www.ualberta.ca/$\sim$xzhuang}}


\makeatletter \@addtoreset{equation}{section} \makeatother

\begin{abstract} Let $(\pP,\wt\pP)$ be a pair of  $r \times s$ matrices of Laurent polynomials with
symmetry such that $\pP(z) \wt\pP^*(z)=I_\mrow$ for all $z\in \CC
\bs \{0\}$ and both $\pP$ and $\wt\pP$ have the same symmetry
pattern that is compatible. The biorthogonal matrix extension
problem with symmetry is to find a pair of  $s \times s$ square
matrices $(\pP_e,\wt\pP_e)$ of Laurent polynomials with symmetry
such that $[I_r, \mathbf{0}] \pP_e =\pP$ and
$[I_r,\mathbf{0}]\wt\pP_e=\wt\pP$ (that is, the submatrix of the
first $r$ rows of $\pP_e,\wt\pP_e$ is the given matrix $\pP,\wt\pP$,
respectively), $\pP_e$ and $\wt\pP_e$ are biorthogonal satisfying
$\pP_e(z)\wt\pP_e^*(z)=I_\mcol$ for all $z\in \CC \bs \{0\}$, and
$\pP_e,\wt\pP_e$ have the same compatible symmetry.  In this paper,
we satisfactorily solve this matrix extension problem with symmetry
by constructing the desired pair of  extension matrices
$(\pP_e,\wt\pP_e)$ from the given pair of matrices $(\pP,\wt\pP)$.
 Matrix extension
plays an important role in many areas such as wavelet analysis,
electronic engineering, system sciences, and so on. As an
application of our general results on matrix extension with
symmetry, we obtain a satisfactory algorithm for constructing
 symmetric biorthogonal  multiwavelets by deriving high-pass filters with symmetry from any
given pair of biorthgonal low-pass filters with symmetry. Several
examples of symmetric biorthogonal multiwavelets are provided to
illustrate the results in this paper.
\end{abstract}

\begin{keyword}
Biorthogonal multiwavelets\sep matrix extension\sep filter \sep
filter banks\sep symmetry\sep Laurent polynomials.

\MSC[2000]{42C40, 41A05, 42C15, 65T60}
\end{keyword}


\maketitle \pagenumbering{arabic}

\section{Introduction and Main Result}
The matrix extension problem plays a fundamental role in many areas
such as electronic engineering, system sciences, mathematics, and
etc. To mention only a few references here on this topic, see
\cite{Cen.Cen:2008,Chui.Lian:1995,Cui:2005,Daubechies:book,
Geronimo.Hardin.Massopust:1994,Han:1998,Han:2009:JFAA,Han.Ji:2009,Jiang:2001,Lawton.Lee.Shen:1996,Petukhov:2004,Shen:1998,Vaidyanathan:1992,Youla.Pickel:1984}.
For example, matrix extension is an indispensable tool in the design
of filter banks  in electronic engineering
(\cite{Jiang:2001,Vaidyanathan:1992,Youla.Pickel:1984}) and in the
construction of multiwavelets in wavelet analysis
(\cite{Chui.Lian:1995,Cui:2005,Daubechies:book,Donovan.Geronimo.Hardin:1996,
Geronimo.Hardin.Massopust:1994,Han:1998,Han:2009:JFAA,Ji.Shen:1999,Lawton.Lee.Shen:1996,Petukhov:2004}).
In order to state the biorthogonal matrix extension problem and our
main result on this topic, let us introduce some notation and
definitions first.

Let $\pp(z)=\sum_{k\in \Z} p_k z^k, z\in \CC \bs \{0\}$ be a Laurent
polynomial with complex coefficients $p_k\in \CC$ for all $k\in \Z$.
We say that $\pp$ has {\it symmetry} if its coefficient sequence
$\{p_k\}_{k\in \Z}$ has symmetry; more precisely, there exist $\gep
\in \{-1,1\}$ and $c\in \Z$ such that
\begin{equation}\label{sym:seq}
p_{c-k}=\gep p_k, \qquad \forall\; k\in \Z.
\end{equation}
If $\gep=1$, then $\pp$ is symmetric about the point $c/2$; if
$\gep=-1$, then $\pp$ is antisymmetric about the point $c/2$.
Symmetry of a Laurent polynomial can be conveniently expressed using
a symmetry operator $\sym$ defined by
\begin{equation}\label{sym}
\sym \pp(z):=\frac{\pp(z)}{\pp(1/z)}, \qquad z\in \CC \bs \{0\}.
\end{equation}
When $\pp$ is not identically zero,  it is evident that
\eqref{sym:seq} holds if and only if $\sym \pp(z)=\gep z^c$. For the
zero polynomial, it is very natural that $\sym 0$ can be assigned
 any symmetry pattern; that is, for every occurrence
of $\sym 0$ appearing in an identity in this paper, $\sym 0$  is
understood to take an appropriate choice of $\gep z^c$ for some
$\gep\in \{-1,1\}$ and $c\in \Z$ so that the identity holds. If
$\pP$ is an $r\times s$ matrix of Laurent polynomials with symmetry,
then we can apply the operator $\sym$ to each entry of $\pP$, that
is, $\sym \pP$ is an $r\times s$ matrix such that $[\sym
\pP]_{j,k}:=\sym ([\pP]_{j,k})$, where $[\pP]_{j,k}$ is the
$(j,k)$-entry of the matrix $\pP$. Also
$[\pP]_{j,k:\ell}:=[[\pP]_{j,k},[\pP]_{j,k+1},\ldots,[\pP]_{j,\ell}]$
is a $1\times (\ell-k+1)$ vector.

For two matrices $\pP$ and $\pQ$ of Laurent polynomials with
symmetry, even though all the entries in $\pP$ and $\pQ$ have
symmetry, their sum $\pP+\pQ$, difference $\pP-\pQ$, or product
$\pP\pQ$, if well defined, generally may not have symmetry any more.
This is one of the difficulties for matrix extension with symmetry.
In order for $\pP\pm\pQ$ or $\pP \pQ$ to possess some symmetry, the
symmetry patterns of $\pP$ and $\pQ$ should be compatible. For
example, if $\sym \pP=\sym \pQ$, that is, both $\pP$ and $\pQ$ have
the same symmetry pattern, then indeed $\pP\pm\pQ$ has symmetry and
$\sym(\pP\pm\pQ)=\sym \pP=\sym \pQ$.  In the following, we discuss
the compatibility of symmetry patterns of matrices of Laurent
polynomials. For an $r\times s$ matrix $\pP(z)=\sum_{k\in \Z} P_k
z^k$,  we denote
\begin{equation}\label{Pstar}
\pP^*(z):=\sum_{k\in \Z} P_k^* z^{-k} \quad \mbox{with}
\quad P_k^*:=\ol{P_k}^T,\qquad k\in \Z,
\end{equation}
where $\ol{P_k}^T$ denotes the transpose of the complex conjugate of
the constant matrix $P_k$ in $\CC$. We say that \emph{the symmetry
of $\pP$ is compatible} or \emph{$\pP$ has compatible symmetry}, if
\begin{equation}\label{sym:comp}
\sym \pP(z)=(\sym \pth_1)^*(z) \sym \pth_2(z),
\end{equation}
for some $1 \times r$ and $1\times s$ row vectors $\pth_1$ and
$\pth_2$ of Laurent polynomials with symmetry. For an $r\times s$
matrix $\pP$ and an $s\times t$ matrix $\pQ$ of Laurent polynomials,
we say that $(\pP, \pQ)$  \emph{has mutually compatible symmetry} if
%
\begin{equation}\label{sym:mutual}
\sym \pP(z)=(\sym \pth_1)^*(z) \sym \pth(z) \quad \mbox{and} \quad
\sym \pQ(z)=(\sym \pth)^*(z) \sym \pth_2(z)
\end{equation}
for some $1\times r$, $1\times s$,  $1\times t$ row vectors $\pth_1,
\pth, \pth_2$ of Laurent polynomials with symmetry. If $(\pP,\pQ)$
has mutually compatible symmetry as in \eqref{sym:mutual}, then
their product $\pP \pQ$ has compatible symmetry and in fact $\sym
(\pP\pQ)= (\sym \pth_1)^*\sym \pth_2 $.

For a matrix of Laurent polynomials, another important property is
the support of its coefficient sequence. For $\pP=\sum_{k\in \Z} P_k
z^k$ such that $P_k={\bf0}$ for all $k\in \Z \bs [m,n]$ with
$P_{m}\ne{\bf0}$ and $P_{n}\ne {\bf0}$, we define its coefficient
support to be $\cs(\pP):=[m,n]$ and the length of its coefficient
support to be $|\cs(\pP)|:=n-m$. In particular, we define
$\cs({\bf0}):=\emptyset$, the empty set, and
$|\cs({\bf0})|:=-\infty$. Also, we  use $\coeff(\pP,k):=P_k$ to
denote the coefficient matrix (vector) $P_k$ of $z^k$ in $\pP$.
Throughout this paper, ${\bf0}$ always denotes a general zero matrix
whose size can be determined in the context.  $\mathbf{1}_n$ denotes
the $1\times n$ row vector $[1,\ldots,1]$,

The Laurent polynomials that we shall consider  have their
coefficients in a subfield $\F$ of the complex field $\CC$.
%
%
Several particular examples of such subfields $\F$ are $\F=\Q$ (the
field of rational numbers), $\F=\R$ (the field of real numbers), and
$\F=\CC$ (the field of complex numbers).

%
%

Throughout the paper,   $r$ and $s$ denote two positive integers
such that $1\le r\le s$. Now we generalize the  matrix extension
problem we consider in \cite{Han.Zhuang:MatExt} to the biorthogonal
case as follows: Let $(\pP,\wt\pP)$ be a pair of  $r\times s$
matrices of Laurent polynomials with coefficients in $\F$ such that
$\pP(z) \wt\pP^*(z)=I_{r}$ for all $z\in \CC\backslash\{0\}$, the
symmetry of each $\pP$ and $\wt\pP$ is compatible, and
$\sym\pP=\sym\wt\pP$. Find a pair of
 $s\times s$ square matrices $(\pP_e,\wt\pP_e)$ of Laurent polynomials
with coefficients in $\F$ and with symmetry such that $[I_r,
\mathbf{0}] \pP_e=\pP, [I_r, \mathbf{0}] \wt\pP_e=\wt\pP,$ (that is,
the submatrix of the first $r$ rows of $\pP_e, \wt\pP_e$ is the
given matrix $\pP,\wt\pP$, respectively), the symmetry of $\pP_e$
and $\wt\pP_e$ is compatible, and $\pP_e(z)\wt\pP_e^*(z)=I_{s}$ for
all $z\in \CC \bs \{0\}$. The coefficient support of
$\pP_e,\wt\pP_e$ can be controlled by that of $\pP,\wt\pP$ in some
way.

The above extension problem plays a critical role in wavelet
analysis. The key of wavelet constructions is the so-called
multiresolution analysis (MRA), which contains mainly two parts. One
is on the construction of refinable function vectors that satisfies
certain desired conditions. For example, (bi)orthogonality,
symmetry, regularity, and so on. Another part is on the derivation
of wavelet generators from refinable function vectors obtained in
first part, which should be able to inherit certain properties
similar to their refinable function vectors. From the point of view
of filter banks, the first part corresponds to the design of filters
or filter banks with certain desired properties, while the second
part can be and is formulated as a matrix extension problem given
above. For the construction of  biorthogonal refinable function
vectors (a pair of biorthogonal low-pass filters), the CBC
(\emph{coset by coset}) algorithm proposed in \cite{Han:2001:JAT}
(also see Section~3 for more details) provides a systematic way of
constructing a desirable dual mask from a given primal mask that
satisfies certain conditions. More precisely, given  a mask
(low-pass filter) satisfying the condition that a dual mask exists,
following the CBC algorithm, one can construct a dual mask with any
preassigned orders of sum rules, which is closely related to the
regularity of the refinable function vectors. Furthermore, if the
primal mask has symmetry, then the CBC algorithm also guarantees
that the dual mask has symmetry. Thus, the first part of MRA
corresponding to the construction of biorthogonal multiwavelets is
more or less solved. However, how to derive the wavelet generators
(high-pass filters) with symmetry remains open even for the scalar
case ($r=1$) and this is one of the motivations of this paper. We
shall see that using our extension algorithm, the wavelet generators
do have symmetry once the given refinable function vectors possess
certain symmetry patterns.

Due to the flexibility of biorthogonality $\pP\wt\pP^*=I_r$, the
above extension problem becomes far more complicated than that the
matrix extension problem we considered in \cite{Han.Zhuang:MatExt}.
The difficulty here is not the symmetry patterns of the extension
matrices, but the support control of the extension matrices. Without
considering any issue on support control, almost all results of
Theorems~1 and 2 in \cite{Han.Zhuang:MatExt} can be transferred to
the biorthogonal case without much difficulty. In
\cite{Han.Zhuang:MatExt}, we showed that the length of the
coefficient support of the extension matrix can never exceed the
length of the coefficient support of the given matrix. Yet, for the
extension matrices in the biorthogonal extension case, we can no
longer expect such nice result, that is, in this case, the length of
the coefficient supports of the extension matrices might not  be
controlled by one of the given matrices. Let us present an example
here to show why we might not have such a result.

\begin{example}\label{ex:lengthBioExt}
{\rm Consider two $1\times 3$ vectors of Laurent polynomials
$\pp(z)=[1,0,a(z)]$ and $\wt\pp(z)=[1,\wt a(z),0]$ with
$|\cs(a(z))|>0, |\cs(\wt a(z))|>0$. We have $\pp\wt\pp^*=1$. Let
$\pP_e$ and $\wt\pP_e$ be their extension matrices such that
$\pP_e\wt\pP_e^*=I_3$. Then $\pP_e, \wt\pP_e$ must be of the form:
\[
\pP_e=\left[
        \begin{array}{ccc}
          1 & 0 & a(z) \\
          -b_1(z)\wt a^*(z) & b_1(z) & c_1(z) \\
          -b_2(z)\wt a^*(z) & b_2(z) & c_2(z) \\
        \end{array}
      \right],
\wt\pP_e=\left[
        \begin{array}{ccc}
          1 & \wt a(z) & 0 \\
          -\wt c_1(z)a^*(z) & \wt b_1(z) & \wt c_1(z) \\
          -\wt c_2(z)a^*(z) & \wt b_2(z) & \wt c_2(z) \\
        \end{array}
      \right].
\]

It is easy to show that $\det(\pP_e)=b_1(z)c_2(z)-b_2(z)c_1(z)$.
Since $\pP_e$ is invertible with $\pP_e^{-1}=\wt\pP_e^*$, we know
that $\det(\pP_e)$ must be a monomial. Without loss of generality,
we can assume $b_1(z)c_2(z)-b_2(z)c_1(z)=1$. Using the cofactors of
$\pP_e$, it is easy to show that $\wt\pP_e=(\pP_e^{-1})^*$ must be
of the form:
\[
\wt\pP_e=\left[
        \begin{array}{ccc}
          1 & \wt a(z) & 0 \\
          b_2^*(z)a^*(z) & c_2^*(z)+\wt a(z) a^*(z)b_2^*(z) &  -b_2^*(z) \\
          -b_1^*(z)a^*(z) & -c_1^*(z)-\wt a(z) a^*(z)b_1^*(z) &  b_1^*(z) \\
        \end{array}
      \right].
\]

On the one hand, if $|\cs(b_1(z))|>0$ or $|\cs(b_2(z))|>0$, then we
see that one of the extension matrices will have  support length
exceeding the maximal length of the given columns. One the other
hand, if both $|\cs(b_1(z))|=0$ and $|\cs(b_2(z))|=0$ (in this case,
both $b_1(z)$ and $b_2(z)$ are monomials), then the lengths of the
coefficient support of $c_1(z)$ and $c_2(z)$ in $\wt\pP_e$ must be
comparable with $\wt a^*(z)a(z)$ so that the support length of
$\wt\pP_e$ can be controlled by that of $\pp$ or $\wt\pp$, which in
turn will result in longer support length of $\pP_e$.
 }
\end{example}

The above example shows that it is difficult to control the support
length of the coefficient support of the extension matrices
independently by only one  given vector in the biorthogonal setting.
Nevertheless, we have the following result, which indicate the
lengths of the coefficient support of the extension matrices can be
controlled by the given pair in certain sense.

\begin{theorem}\label{bio:thm:main:1}
Let $\F$ be a subfield of $\CC$. Let $(\pP,\wt\pP)$ be a pair of
$r\times s$ matrices of Laurent polynomials with coefficients in
$\F$ such that the symmetry of each $\pP,\wt\pP$ is compatible:
$\sym\pP=\sym\wt\pP=(\sym\pth_1)^*\sym\pth_2$ for some $1\times r$,
$1\times s$ vectors $\pth_1,\pth_2$ of Laurent polynomials with
symmetry. Moreover, $\pP(z)\wt\pP^*(z)=I_r$ for all $z\in \CC \bs
\{0\}$. Then there exists a pair of $s \times s$ square matrices
$(\pP_e,\wt\pP_e)$ of Laurent polynomials with coefficients in $\F$
such that
\begin{enumerate}
\item[{\rm(i)}] $[I_r,{\bf0}] \pP_e=\pP, [I_r,{\bf0}]\wt\pP_e=\wt\pP$,
that is, the submatrices of the first  $r$ rows of $\pP_e,\wt\pP_e$ are $\pP,\wt\pP$, respectively;

\item[{\rm(ii)}] $\pP_e$ and $\wt\pP_e$ are biorthogonal: $\pP_e(z) \wt\pP_e^*(z)=I_s$ for all $z\in \CC \bs \{0\}$;

\item[{\rm(iii)}] The symmetry of each $\pP_e,\wt\pP_e$ is
compatible: $\sym\pP_e=\sym\wt\pP_e=(\sym\pth)^*\sym\pth_2$ for some
$1\times s$ vector $\pth$ of Laurent polynomials with symmetry.

\item[{\rm{(iv)}}] $\pP_e, \wt\pP_e$ can be represented as:
\begin{equation}\label{cascade:bior}
\begin{aligned}
\pP_e(z)=\pP_J(z)\cdots\pP_1(z),\quad
\wt\pP_e(z)=\wt\pP_J(z)\cdots\wt\pP_1(z),
\end{aligned}
\end{equation}
where $\pP_j,\wt\pP_j, 1\le j\le J$ are $s\times s$ matrices of
Laurent polynomials with symmetry that satisfy
$\pP_j(z)\wt\pP_j^*(z)=I_s$. Moreover, each pair of
$(\pP_{j+1},\pP_j)$ and $(\wt\pP_{j+1},\wt\pP_j)$ has mutually
compatible symmetry for all $j=1,\ldots, J-1$.

\item[{\rm(v)}] If $r=1$, then the coefficient supports of $\pP_e,\wt\pP_e$ are controlled by that of $\pP,\wt\pP$ in the following
sense:
\begin{equation}\label{bio:supp:control}
\begin{small}
\begin{aligned}
\max_{1\le j,k\le s}\{|\cs([\pP_e]_{j,k})|,|\cs([\wt\pP_e]_{j,k})|\}
&\le \max_{1\le\ell\le
s} |\cs([\pP]_\ell)|+\max_{1\le \ell\le s} |\cs([\wt\pP]_\ell)|.\\
\end{aligned}
\end{small}
\end{equation}
\end{enumerate}
\end{theorem}

For $r=1$, Goh et al. in \cite{Goh.Yap:1998} considered this matrix
extension problem without symmetry. They provided a step-by-step
algorithm for deriving the extension matrices, yet they  did not
concern about the support control of the extension matrices nor the
symmetry patterns of the extension matrices. For $r>1$, there are
only a few results in the literature \cite{Cen.Cen:2008,
Cui:2005:AMC} and most of them  concern only about some very special
cases. The difficulty still comes from the flexibility of the
biorthogonality relation between the given two matrices. In this
paper, we shall mainly consider this matrix extension problem with
symmetry for the biorthogonal case and shall provide an extension
algorithm from which the extension matrices can have both symmetry
and support control as stated in Theorem \ref{bio:thm:main:1}.

 Here is the
structure of this paper. In Section~2, we shall introduce some
auxiliary results, prove Theorem \ref{bio:thm:main:1}, and also
provide a step-by-step algorithm for the construction of the
extension matrices. In Section~3, we shall discuss the applications
of our main result to the construction of symmetric biorthogonal
multiwavelets in wavelet analysis. Examples will be provided to
illustrate our algorithms. Conclusions and remarks shall be given in
the last section.

\section{Proof of Theorem~\ref{bio:thm:main:1} and an  Algorithm}
In this section, we shall prove our main result
Theorem~\ref{bio:thm:main:1} and based on the the proof, we shall
provide a step-by-step extension algorithm for deriving the desired
pair of extesion matrices.

First, let us introduce some auxiliary results. The following lemma
shows that for a pair of constant vector $(\vf,\wt\vf)$ in $\F$, we
can find a pair of biorthogonal matrices $\left(U_{(\vf,\wt\vf)},\wt
U_{(\vf,\wt\vf)}\right)$ such that up to a constant multiplication,
they normalize $\vf,\wt\vf$ to two unit vectors, respectively.

\begin{lemma}\label{lemma:v.u.A.B}
Let $(\vf,\wt\vf)$ be a pair of nonzero $1\times n$ vectors in $\F$.
Then the following statements hold.
\begin{itemize}
\item[{\rm (1)}] If $\vf\wt\vf^*\neq 0$, then there exists a pair of
$n\times n$ matrices $\left(U_{(\vf,\wt\vf)}, \wt
U_{(\vf,\wt\vf)}\right)$ in $\F$ such that
$U_{(\vf,\wt\vf)}=[(\frac{\wt\vf}{\wt c})^*,F]$, $\wt
U_{(\vf,\wt\vf)}=[(\frac{\vf}{c})^*, \wt F]$, and
$U_{(\vf,\wt\vf)}\wt U_{(\vf,\wt\vf)}^*=I_n$, where $F,\wt F$ are
$n\times (n-1)$ constant matrices in $\F$ and $c,\wt c$ are two
nonzero numbers in $\F$ such that $\vf\wt\vf^*=c\overline{\wt c}$.
In this case, $\vf U_{(\vf,\wt\vf)}=c\e_1$ and $\wt\vf \wt
U_{(\vf,\wt\vf)}=\wt c\e_1$.

\item[{\rm (2)}]If $\vf\wt\vf^*= 0$, then there exists a pair of
$n\times n$ matrices $\left(U_{(\vf,\wt\vf)}, \wt
U_{(\vf,\wt\vf)}\right)$ in $\F$ such that
$U_{(\vf,\wt\vf)}=[(\frac{\vf}{\wt c_1})^*, (\frac{\wt\vf}{ c_2})^*,
F]$, $\wt U_{(\vf,\wt\vf)}=[(\frac{\vf}{ c_1})^*, (\frac{\wt\vf}{\wt
c_2})^*, \wt F]$, and $U_{(\vf,\wt\vf)}\wt U_{(\vf,\wt\vf)}^*=I_n$,
where $F,\wt F$ are $n\times (n-2)$ constant matrices  in $\F$ and
$c_1, c_2,\wt c_1,\wt c_2$ are nonzero numbers in $\F$ such that
$\|\vf\|^2=c_1\overline{\wt c_1}$,$\|\wt\vf\|^2=c_2\overline{\wt
c_2}$. In this case, $\vf U_{(\vf,\wt\vf)}=c_1\e_1$ and $\wt\vf \wt
U_{(\vf,\wt\vf)}=c_2\e_2$.
\end{itemize}
\end{lemma}

\begin{proof} If $\vf\wt\vf^*\neq0$, there exists $\{\vf_2,\ldots,\vf_n\}$  being a basis of the
orthogonal compliment of the linear span of $\{\vf\}$ in $\F^n$. Let
$F:=[\vf_2^*,\ldots,\vf_n^*]$ and
$U_{(\vf,\wt\vf)}:=[(\frac{\wt\vf}{\wt c})^*,F]$. Then
$U_{(\vf,\wt\vf)}$ is invertible. Let $\wt
U_{(\vf,\wt\vf)}:=\left(U_{(\vf,\wt\vf)}^{-1}\right)^*$. It is easy
to show that $U_{(\vf,\wt\vf)}$ and $\wt U_{(\vf,\wt\vf)}$ are the
desired matrices.

If $\vf\wt\vf^*=0$, let $\{\vf_3,\ldots,\vf_n\}$  be a basis of the
orthogonal compliment of the linear span of $\{\vf,\wt\vf\}$ in
$\F^n$. Let $U_{(\vf,\wt\vf)}=[(\frac{\vf}{\wt c_1})^*,
(\frac{\wt\vf}{ c_2})^*,  F]$ with $F:=[\vf_3^*,\ldots,\vf_n^*]$.
Then $U_{(\vf,\wt\vf)}$ and $\wt
U_{(\vf,\wt\vf)}:=\left(U_{(\vf,\wt\vf)}^{-1}\right)^*$ are the
desired matrices. \eop
\end{proof}

Thanks to Lemma~\ref{lemma:v.u.A.B}, we can reduce the support
lengths of a pair $(\pp,\wt\pp)$ of Laurent polynomials with
symmetry by constructing a pair of biorthogonal matrices
$(\pB,\wt\pB)$ of Laurent polynomials with symmetry as stated in the
following lemma.
\begin{lemma}\label{lemma:bioDegBy1}
Let $(\pp,\wt\pp)$ be a pair of $1\times s$ vectors of Laurent
polynomials with symmetry such that $\pp\wt\pp^*=1$ and
$\sym\pp=\sym\wt\pp=\gep z^c[{\bf 1}_{s_1},-{\bf 1}_{s_2},z^{-1}{\bf
1}_{s_3},-z^{-1}{\bf 1}_{s_4}]=:\sym\pth$ for some nonnegative
integers $s_1,\ldots,s_4$ satisfying $s_1+\cdots+s_4=s$ and
$\gep\in\{1,-1\}, c\in \{0,1\}$. Suppose $|\cs(\pp)|>0$. Then there
exist a pair of $s\times s$ matrices $(\pB,\wt\pB)$ of Laurent
polynomials with symmetry such that
\begin{itemize}
\item[{\rm (1)}] $\pB,\wt\pB$ are biorthogonal:
$\pB(z)\wt\pB^*(z)=I_n$;

\item[{\rm (2)}] $\sym\pB=\sym\wt\pB=(\sym\pth)^*\sym\pth_1$ with
$\sym\pth_1=\gep z^c[{\bf 1}_{s_1'},-{\bf 1}_{s_2'},z^{-1}{\bf
1}_{s_3'},-z^{-1}{\bf 1}_{s_4'}]$ for some nonnegative integers
$s_1',\ldots,s_4'$ such that $s_1'+\cdots+s_4'=s$;

\item[{\rm (3)}] the length of the coefficient support of
$\pp$ is reduced by that of $\pB$. $\wt\pB$ does not increase the
length of the coefficient support of $\wt\pp$. That is,
$|\cs(\pp\pB)|\le|\cs(\pp)|-|\cs(\pB)|$ and
$|\cs(\wt\pp\wt\pB)|\le|\cs(\wt\pp)|$.
\end{itemize}
\end{lemma}

\begin{proof} We shall only prove  the case that $\sym\pth = [{\bf 1}_{s_1},-{\bf 1}_{s_2},z^{-1}{\bf 1}_{s_3},-z^{-1}{\bf
1}_{s_4}]$. The proofs for other cases are similar. By their
symmetry patterns, $\pp$ and $\wt\pp$ must take the form as follows
with $\ell>0$ and $\coeff(\pp,-\ell)\neq{\bf0}$:
\begin{equation}\label{eq:p}
\begin{aligned}
\pp&=[\vf_1, -\vf_2,  \vg_1, -\vg_2]z^{-\ell} +
[\vf_3,-\vf_4,\vg_3,-\vg_4 ]z^{-\ell+1}
+\sum_{k=-\ell+2}^{\ell-2}\coeff(\pp,k)z^k
\\&+ [\vf_3,\vf_4,\vg_1,\vg_2]z^{\ell-1} +
[\vf_1,\vf_2,\textbf{0},{\bf0}]z^{\ell};
\\
\wt\pp&=[\wt\vf_1, -\wt\vf_2,  \wt\vg_1, -\wt\vg_2]z^{-\wt\ell} +
[\wt\vf_3,-\wt\vf_4,\wt\vg_3,-\wt\vg_4 ]z^{-\wt\ell+1}
+\sum_{k=-\wt\ell+2}^{\wt\ell-2}\coeff(\wt\pp,k)z^k
\\&+ [\wt\vf_3,\wt\vf_4,\wt\vg_1,\wt\vg_2]z^{\wt\ell-1} +
[\wt\vf_1,\wt\vf_2,\textbf{0},{\bf0}]z^{\wt\ell};
\end{aligned}
\end{equation}
Then, either $\|\vf_1\|+\|\vf_2\|\neq0$ or
$\|\vg_1\|+\|\vg_2\|\neq0$.  Considering $\|\vf_1\|+\|\vf_2\|\neq0$,
due to $\pp\wt\pp^*=1$ and $|\cs(\pp)|>0$, we have
$\vf_1\wt\vf_1^*-\vf_2\wt\vf_2^*=0$. Let
$c:=\vf_1\wt\vf_1^*=\vf_2\wt\vf_2^*$. Then there are at most three
cases: (a) $c\neq0$; (b) $c=0$ but both $\vf_1,\vf_2$ are nonzero
vectors; (c) $c=0$ and one of $\vf_1,\vf_2$ is  ${\bf0}$.

Case (a): In this case, we have $\vf_1\wt\vf_1^* \neq 0$ and
$\vf_2\wt\vf_2\neq 0$. By Lemma \ref{lemma:v.u.A.B}, we can
construct two pairs of biorthogonal matrices
$\left(U_{(\vf_1,\wt\vf_1)},\wt U_{(\vf_1,\wt\vf_1)}\right)$ and
$\left(U_{(\vf_2,\wt\vf_2)},\wt U_{(\vf_2,\wt\vf_2)}\right)$ with
respect to the pairs $(\vf_1,\wt\vf_1)$ and $(\vf_2,\wt\vf_2)$ such
that
\[
\begin{aligned}
U_{(\vf_1,\wt\vf_1)}&=\left[\left(\frac{\wt\vf_1}{\wt c_1}\right)^*,
F_1\right], & \wt U_{(\vf_1,\wt\vf_1)}&=\left[\left(\frac{\vf_1}{
c_1}\right)^*, \wt F_1\right], & \vf_1
U_{(\vf_1,\wt\vf_1)}&=c_1\e_1,
& \wt \vf_1 \wt U_{(\vf_1,\wt\vf_1)}&=\wt c_1\e_1, \\
U_{(\vf_2,\wt\vf_2)}&=\left[\left(\frac{\wt\vf_2}{\wt c_1}\right)^*,
F_2\right], &\wt
 U_{(\vf_2,\wt\vf_2)}&=\left[\left(\frac{\vf_2}{ c_1}\right)^*, \wt F_2\right],
 &  \vf_2
U_{(\vf_2,\wt\vf_2)}&=c_1\e_1, &\wt \vf_2 \wt
U_{(\vf_2,\wt\vf_2)}&=\wt c_1\e_1,
\end{aligned}
\]
where $c_1,\wt c_1$ are constants in $\F$ such that
$c=c_1\overline{\wt c_1}$. Define $\pB_0(z),\wt\pB_0(z)$ as follows:
\begin{equation}\label{bio:B01}
\begin{aligned}
\pB_0(z)&=\left[
           \begin{array}{cc|cc|c}
             \frac{1+z^{-1}}{2}(\frac{\wt\vf_1}{\wt c_1})^* & F_1 &  -\frac{1-z^{-1}}{2}(\frac{\wt\vf_1}{\wt c_1})^* & {\bf0} & {\bf 0} \\
             -\frac{1-z^{-1}}{2}(\frac{\wt\vf_2}{\wt c_1})^* & {\bf0} &  \frac{1+z^{-1}}{2}(\frac{\wt\vf_2}{\wt c_1})^* & F_2 & {\bf 0} \\
             \hline
             {\bf0} &  {\bf0}  &  {\bf0}  &  {\bf0}  & I_{s_3+s_4} \\
           \end{array}
         \right],\\
\wt \pB_0(z)&=\left[
           \begin{array}{cc|cc|c}
             \frac{1+z^{-1}}{2}(\frac{\vf_1}{ c_1})^* & \wt F_1 &  -\frac{1-z^{-1}}{2}(\frac{\vf_1}{ c_1})^* & {\bf0}  & {\bf 0} \\
             -\frac{1-z^{-1}}{2}(\frac{\vf_2}{ c_1})^* & {\bf0} &  \frac{1+z^{-1}}{2}(\frac{\vf_2}{ c_1})^* & \wt F_2& {\bf 0} \\
             \hline
             {\bf0} &  {\bf0}  &  {\bf0}  &  {\bf0}  & I_{s_3+s_4} \\
           \end{array}
         \right].
\end{aligned}
\end{equation}
Direct computation shows that $\pB_0(z)\wt\pB_0(z)^*=I_s$  due to
the special structures of the pairs $\left(U_{(\vf_1,\wt\vf_1)},\wt
U_{(\vf_1,\wt\vf_1)}\right)$ and $\left(U_{(\vf_2,\wt\vf_2)},\wt
U_{(\vf_2,\wt\vf_2)}\right)$ constructed by
Lemma~\ref{lemma:v.u.A.B}. The symmetry patterns of $\pp\pB_0$ and
$\wt\pp\wt\pB_0$ satisfies
\[
\sym(\pp\pB_0)=\sym(\wt\pp\wt\pB_0)=[z^{-1},{\bf
1}_{s_1-1},-z^{-1},-{\bf 1}_{s_2-1},z^{-1}{\bf 1}_{s_3},-z^{-1}{\bf
1}_{s_4}].
\]
Moreover, $\pB_0(z)$, $\wt\pB_0(z)$ reduce the lengths of the
coefficient support of $\pp,\wt\pp$ by $1$, respectively.

In fact, due to the above symmetry pattern and the structures of
$\pB_0,\wt\pB_0$, we only need to show that
$\coeff([\pp\pB_0]_j,\ell)=\coeff([\wt\pp\wt\pB_0]_j,\ell)=0$ for
$j=1,s_1+1$. Note that
$\coeff([\pp\pB_0]_j,\ell)=\coeff(\pp,\ell)\coeff([\pB_0]_{:,1},0)=\frac{1}{2\overline{\wt
c_1}}(\vf_1 \wt\vf_1^*-\vf_2\wt\vf_2^*)=0$. Similar computations
apply for other terms. Thus, $|\cs(\pp\pB_0)|<\cs(\pp)$ and
$|\cs(\wt\pp\wt\pB_0)|<|\cs(\wt\pp)|$. Let $E$ be a permutation
matrix such that
\[
\sym(\pp\pB_0) E=\sym(\wt\pp\wt\pB_0) E=[{\bf 1}_{s_1-1},-{\bf
1}_{s_2-1},z^{-1}{\bf 1}_{s_3+1},-z^{-1}{\bf
1}_{s_4+1}]=:\sym\pth_1.
\]
Define $\pB(z)=\pB_0(z)E$ and $\wt\pB(z)=\wt\pB_0(z)E$. Then
$\pB(z)$ and $\wt\pB(z)$ are the desired matrices.

Case (b): In this case, $\vf_1\wt\vf_1^*=\vf_2\wt\vf_2^*=0$ and both
$\vf_1,\vf_2$ are nonzero vectors. We have $\vf_1\vf_1^*\neq 0$ and
$\vf_2\vf_2^*\neq0$. Again, by Lemma~\ref{lemma:v.u.A.B}, we can
construct two pairs of biorthogonal matrices
$\left(U_{(\vf_1,\vf_1)},\wt U_{(\vf_1,\vf_1)}\right)$ and
$\left(U_{(\vf_2,\vf_2)},\wt U_{(\vf_2,\vf_2)}\right)$ with respect
to the pairs $(\vf_1,\vf_1)$ and $(\vf_2,\vf_2)$ such that
\[
\begin{aligned}
U_{(\vf_1,\vf_1)}&=\left[\left(\frac{\vf_1}{\wt c_1}\right)^*,
F_1\right], & \wt U_{(\vf_1,\vf_1)}&=\left[\left(\frac{\vf_1}{
c_0}\right)^*, F_1\right], & \vf_1
U_{(\vf_1,\vf_1)}&=c_0\e_1, \\
U_{(\vf_2,\vf_2)}&=\left[\left(\frac{\vf_2}{\wt c_2}\right)^*,
F_2\right], &\wt
 U_{(\vf_2,\vf_2)}&=\left[\left(\frac{\vf_2}{ c_0}\right)^*,  F_2\right],
 &  \vf_2
U_{(\vf_2,\vf_2)}&=c_0\e_1,
\end{aligned}
\]
where $c_0,\wt c_1,\wt c_2$ are constants in $\F$ such that
$\vf_1\vf_1^*=c_0\overline{\wt c_1}$ and $\vf_2\vf_2^* =
c_0\overline{\wt c_2}$. Let $\pB_0,\wt\pB_0(z)$ be defined as
follows:
\begin{equation}\label{bio:B02}
\begin{aligned}
\pB_0(z)&=\left[
           \begin{array}{cc|cc|c}
             \frac{1+z^{-1}}{2}(\frac{\vf_1}{\wt c_1})^* & F_1 &  -\frac{1-z^{-1}}{2}(\frac{\vf_1}{\wt c_1})^* & {\bf0} & {\bf 0} \\
             -\frac{1-z^{-1}}{2}(\frac{\vf_2}{\wt c_2})^* & {\bf0} &  \frac{1+z^{-1}}{2}(\frac{\vf_2}{\wt c_2})^* & F_2 & {\bf 0} \\
             \hline
             {\bf0} &  {\bf0}  &  {\bf0}  &  {\bf0}  & I_{s_3+s_4} \\
           \end{array}
         \right],\\
\wt \pB_0(z)&=\left[
           \begin{array}{cc|cc|c}
             \frac{1+z^{-1}}{2}(\frac{\vf_1}{  c_0})^* & F_1 &  -\frac{1-z^{-1}}{2}(\frac{\vf_1}{ c_0})^* & {\bf0}  & {\bf 0} \\
             -\frac{1-z^{-1}}{2}(\frac{\vf_2}{  c_0})^* & {\bf0} &  \frac{1+z^{-1}}{2}(\frac{\vf_2}{ c_0})^* &  F_2& {\bf 0} \\
             \hline
             {\bf0} &  {\bf0}  &  {\bf0}  &  {\bf0}  & I_{s_3+s_4} \\
           \end{array}
         \right].
\end{aligned}
\end{equation}
We can show that $\pB_0(z)$ reduces the length of the coefficient
support of $\pp$ by 1, while $\wt\pB_0(z)$ does not increase the
support length of $\wt\pp$. Moreover, similar to case (a), we can
find a permutation matrix $E$ such that
\[
\sym(\pp\pB_0) E=\sym(\wt\pp\wt\pB_0) E=[{\bf 1}_{s_1-1},-{\bf
1}_{s_2-1},z^{-1}{\bf 1}_{s_3+1},-z^{-1}{\bf
1}_{s_4+1}]=:\sym\pth_1.
\]
Define $\pB(z)=\pB_0(z)E$ and $\wt\pB(z)=\wt\pB_0(z)E$. Then
$\pB(z)$ and $\wt\pB(z)$ are the desired matrices.

Case (c): In this case, $\vf_1\wt\vf_1^*=\vf_2\wt\vf_2^*=0$ and one
of $\vf_1$ and $\vf_2$ is nonzero.  Without loss of generality, we
assume that $\vf_1\neq {\bf0}$ and $\vf_2={\bf0}$. Construct a pair
of matrices $\left(U_{(\vf_1,\wt\vf_1)},\wt
U_{(\vf_1,\wt\vf_1)}\right)$ by Lemma~\ref{lemma:v.u.A.B} such that
$\vf_1 U_{(\vf_1,\wt\vf_1)}=c_1 \e_1$ and $\wt\vf_1 \wt
U_{(\vf_1,\wt\vf_1)}=c_2 \e_2$ (when $\wt\vf_1={\bf0}$, the pair of
matrices is given by $(U_{(\vf_1,\vf_1)},\wt U_{(\vf_1,\vf_1)})$).
Extend this pair to a pair of  $s\times s$ matrices $(U,\wt{U})$ by
$U:=\diag\left(U_{(\vf_1,\wt\vf_1)},I_{s_3+s_4}\right)$ and $\wt
U:=\diag\left(\wt U_{(\vf_1,\wt\vf_1)},I_{s_3+s_4}\right)$. Then
$\pp U$ and $\wt\pp \wt U$ must be of the form:
%
\[
\begin{aligned}
\pq:=\pp U&=[c_1,0,\ldots,0, -\vf_2,  \vg_1, -\vg_2]z^{-\ell} +
[\vf_3,-\vf_4,\vg_3,-\vg_4 ]z^{-\ell+1}
\\&+\sum_{k=-\ell+2}^{\ell-2}\coeff(\pq,k)z^k
+ [\vf_3,\vf_4,\vg_1,\vg_2]z^{\ell-1} +
[c_1,0,\ldots,0,\vf_2,\textbf{0},{\bf0}]z^{\ell};\\
\end{aligned}
\]
\[
\begin{aligned}
\wt\pq:=\wt\pp \wt U&=[0,c_2,\ldots,0, -\wt\vf_2,  \wt\vg_1,
-\wt\vg_2]z^{-\wt\ell} + [\wt\vf_3,-\wt\vf_4,\wt\vg_3,-\wt\vg_4
]z^{-\wt\ell+1}
\\&+\sum_{k=-\wt\ell+2}^{\wt\ell-2}\coeff(\wt\pq,k)z^k
+ [\wt\vf_3,\wt\vf_4,\wt\vg_1,\wt\vg_2]z^{\wt\ell-1} +
[0,c_2,\ldots,0,,\wt\vf_2,\textbf{0},{\bf0}]z^{\wt\ell};
\end{aligned}
\]
%
If $[\wt\pq]_1\equiv 0$, we choose $k$ such that $k = \mathrm{arg}
\min _{\ell\neq 1}\{|\cs([\pq]_1)|-|\cs([\pq]_\ell)|\}$, i.e., $k$
is an integer such that  the length of coefficient support of
$|\cs([\pq]_1)|-|\cs([\pq]_k)|$ is minimal among those of all
$|\cs([\pq]_1)|-|\cs([\pq]_\ell)|$, $\ell=2,\ldots,s$; otherwise,
due to $\pq\wt\pq^* =0$, there must exist $k$ such that
\[
|\cs([\pq]_1)|-|\cs([\pq]_k)|\le \max_{2\le j\le
s}|\cs([\wt\pq]_j)|-|\cs([\wt\pq]_1)|,
\]
($k$ might not be unique, we can choose one of such $k$ so that
$|\cs([\pq]_1)|-|\cs([\pq]_k)|$ is minimal among all
$|\cs([\pq]_1)|-|\cs([\pq]_\ell)|, \ell =2,\ldots,s$).

For such $k$ (in the  case of either  $[\wt\pq]_1=0$ or
$[\wt\pq]_1\neq 0$), define two matrices $\pB(z),\wt\pB(z)$ as
follows:
\[
\pB(z)=\left[
         \begin{array}{cccc|c}
           1 & 0 &\cdots  & 0 &  \\
           0 & 1 & \cdots & 0 &  \\
           \vdots & \vdots & \ddots &  \vdots&  \\
           -a(z) & 0 &\cdots  & 1 &  \\
           \hline
            &  &  &  & I_{s-k} \\
         \end{array}
       \right],
\wt\pB(z)=\left[
         \begin{array}{cccc|c}
           1 & 0 &\cdots  & a(z)^* &  \\
           0 & 1 & \cdots & 0 &  \\
           \vdots & \vdots & \ddots &  \vdots&  \\
          0 & 0 &\cdots  & 1 &  \\
           \hline
            &  &  &  & I_{s-k} \\
         \end{array}
       \right],
\]
where $a(z)$ in $\pB(z),\wt\pB(z)$ is a Laurent polynomial with
symmetry such that $\sym a(z)=\sym([\pq]_1)/\sym([\pq]_k)$,
$|\cs([\pq]_1-a(z)[\pq]_k)|<|\cs([\pq]_k)|$,  and
$|\cs([\wt\pq]_k-a(z)^*[\wt\pq]_1)|\le\max_{1\le\ell\le
s}|\cs([\wt\pq]_\ell)|$. Such $a(z)$ can be easily obtained by long
division.

It is straightforward  to show that $\pB(z)\wt\pB^*(z)=I_s$.
$\pB(z)$ reduces the length of the coefficient support of $\pq$ by
that of $a(z)$ due to $|\cs([\pq]_1-a(z)[\pq]_k)|<|\cs([\pq]_k)|$.
And by our choice of $k$, $\wt\pB(z)$ does not increase the length
of the coefficient support of $\wt\pq$. Moreover, the symmetry
patterns of both $\pq$ and $\wt\pq$ are preserved.

In summary, for all cases (a), (b), and (c), we can always find a
pair of biorthogonal matrices $(\pB,\wt\pB)$ of Laurent polynomials
such that $\pB$ reduces the length of the coefficient support of
$\pp$ while $\wt\pB$ does not increase the length of the coefficient
support of $\wt\pp$.

For $\|\vf_1\|+\|\vf_2\|=0$, we must have $\|\vg_1\|+\|\vg_2\|\neq
0$. The discussion for this case is similar to above. We can find
two matrices $\pB(z),\wt\pB(z)$ such that all items in the lemma
hold. In the case that
$\vg_1\wt\vg_1^*=\vg_2\wt\vg_2^*=c_1\overline{\wt c_1}\neq 0$, the
pair $(\pB_0(z),\wt\pB_0(z))$ similar to \eqref{bio:B01} is of the
form:
\begin{equation}\label{bio:B02}
\begin{aligned}
\pB_0(z)&=\left[
           \begin{array}{c|cc|cc}
            I_{s_1+s_2} &  {\bf0}  &  {\bf0}  &  {\bf0}  &  {\bf0} \\
            \hline
             {\bf 0}& \frac{1+z}{2}(\frac{\wt\vg_1}{\wt c_1})^* & G_1 &  -\frac{1-z}{2}(\frac{\wt\vg_1}{\wt c_1})^* & {\bf0}  \\
              {\bf 0}&-\frac{1-z}{2}(\frac{\wt\vg_2}{\wt c_1})^* & {\bf0} &  \frac{1+z}{2}(\frac{\wt\vg_2}{\wt c_1})^* & G_2  \\
               \end{array}
         \right],\\
\wt \pB_0(z)&=\left[
           \begin{array}{c|cc|cc}
                I_{s_1+s_2} &  {\bf0}  &  {\bf0}  &  {\bf0}  &  {\bf0} \\
            \hline
              {\bf 0}&\frac{1+z}{2}(\frac{\vg_1}{ c_1})^* & \wt G_1 &  -\frac{1-z}{2}(\frac{\vg_1}{ c_1})^* & {\bf0}   \\
              {\bf 0}&-\frac{1-z}{2}(\frac{\vg_2}{ c_1})^* & {\bf0} &  \frac{1+z}{2}(\frac{\vg_2}{ c_1})^* & \wt G_2 \\
                    \end{array}
         \right].
\end{aligned}
\end{equation}
The pairs  for other cases  can be obtained similarly. We are done.
\eop
\end{proof}

Let $\pth$ be a $1\times n$ row vector of Laurent polynomials with
symmetry such that $\sym\pth=[\gep_1 z^{c_1}, \ldots, \gep_n
z^{c_n}]$ for some $\gep_1, \ldots, \gep_n\in \{-1,1\}$ and $c_1,
\ldots, c_n\in \Z$.  Then, the symmetry of any entry in the vector
$\pth \mbox{diag}(z^{-\lceil c_1/2\rceil}, \ldots, z^{-\lceil
c_n/2\rceil})$ belongs to $\{ \pm 1, \pm z^{-1}\}$. Thus, there is a
permutation matrix $E_\pth$ to regroup these four types of
symmetries together so that
\begin{equation}\label{sym:reorganize}
\sym(\pth \pU_{\sym\pth})=[\mathbf{1}_{n_1}, -\mathbf{1}_{n_2},
z^{-1} \mathbf{1}_{n_3}, -z^{-1}\mathbf{1}_{n_4}],
\end{equation}
where $\pU_{\sym\pth}:= \mbox{diag}(z^{-\lceil c_1/2\rceil}, \ldots,
z^{-\lceil c_n/2\rceil})E_\pth$  and  $n_1$, $\ldots$, $n_4$ are
nonnegative integers uniquely determined by $\sym\pth$.

For an $r\times s$ matrix $\pP$ of Laurent polynomials with
compatible symmetry as in \eqref{sym:comp}, it is easy to see that
$\pQ:=\pU_{\sym \pth_1}^*\pP \pU_{\sym \pth_2}$ has the symmetry
pattern as follows.
\begin{equation}\label{bio:Q:sym:standard}
\sym \pQ=[ {\bf1}_{ r_1},-{\bf1}_{ r_2 }, z{\bf1}_{ r_3 },
-z{\bf1}_{ r_4} ]^T[{\bf1}_{ s_1},-{\bf1}_{ s_2},  z^{-1}{\bf1}_{
s_3},-z^{-1}{\bf1}_{ s_4}].
\end{equation}
Note that $\pU_{\sym\pth_1}$ and $\pU_{\sym\pth_2}$ do not increase
the length of the coefficient support of $\pP$.

 Now, we can prove Theorem~\ref{bio:thm:main:1}
using Lemma~\ref{lemma:bioDegBy1}.
\begin{proof}[Proof of Theorem~\ref{bio:thm:main:1}] First, we
normalize the symmetry patterns of $\pP$ and $\wt\pP$ to the
standard form as in \eqref{bio:Q:sym:standard}. Let $\pQ :=
\pU_{\sym\pth_1}^*\pP\pU_{\sym\pth_2}$ and $\wt\pQ :=
\pU_{\sym\pth_1}^*\wt\pP\pU_{\sym\pth_2}$ (given $\pth$,
$\pU_{\sym\pth}$ is obtained by \eqref{sym:reorganize}). Then the
symmetry of each row of $\pQ$ or $\wt\pQ$ is of the form $\gep
z^c[{\bf1}_{ s_1},-{\bf1}_{ s_2}, z^{-1}{\bf1}_{
s_3},-z^{-1}{\bf1}_{ s_4}]$ for some $\gep\in\{-1,1\}$ and
$c\in\{0,1\}$.

Let $\pp:=[\pQ]_{1,:}$ and $\wt\pp:=[\wt\pQ]_{1,:}$ be the first row
of $\pQ,\wt\pQ$, respectively. Applying Lemma~\ref{lemma:bioDegBy1}
recursively, we can find pairs of biorthogonal matrices of Laurent
polynomials  $(\pB_1,\wt\pB_1)$, ..., $(\pB_K,\wt\pB_K)$ such that
$\pp\pB_1\cdots\pB_K=[1,0,\ldots,0]$ and
$\wt\pp\wt\pB_1\cdots\wt\pB_K=[1,\pq(z)]$ for some $1\times (s-1)$
vector of Laurent polynomials with symmetry. Note that by
Lemma~\ref{lemma:bioDegBy1}, all pairs $(\pB_j,\pB_{j+1})$ and
$(\wt\pB_j,\wt\pB_{j+1})$ for $j=1,\ldots,K-1$ have mutually
compatible symmetry.  Now construct $\pB_{K+1}(z),\wt\pB_{K+1}(z)$
as follows:
\[
\pB_{K+1}(z)=\left[
               \begin{array}{cc}
                 1 & 0 \\
                 \pq^*(z) & I_{s-1} \\
               \end{array}
             \right],  \,\wt\pB_{K+1}(z)=\left[
               \begin{array}{cc}
                 1 & -\pq(z) \\
                {\bf0} & I_{s-1} \\
               \end{array}
             \right].
\]
 $\pB_{K+1}$ and $\wt\pB_{K+1}$ are
biorthogonal. Let $\pA:=\pB_1\cdots\pB_K\pB_{K+1}$ and
$\wt\pA:=\wt\pB_1\cdots\wt\pB_K\wt\pB_{K+1}$. Then
$\pp\pA=\wt\pp\wt\pA=\e_1$.

Note that $\pQ\pA$ and $\wt\pQ\wt\pA$ are of the forms:
\[
\pQ\pA=\left[
         \begin{array}{cc}
           1 & {\bf0} \\
           {\bf0} & \pQ_1(z) \\
         \end{array}
       \right],\,
\wt\pQ\wt\pA=\left[
         \begin{array}{cc}
           1 & {\bf0} \\
           {\bf0} & \wt\pQ_1(z) \\
         \end{array}
       \right]
\]
for some $(r-1)\times s$ matrices $\pQ_1,\wt\pQ_1$ of Laurent
polynomials with symmetry. Moreover, due to
Lemma~\ref{lemma:bioDegBy1}, the symmetry patterns of $\pQ_1$ and
$\wt\pQ_1$ are compatible and satisfies $\sym\pQ_1=\sym\wt\pQ_1$.
The rest of the proof is completed by employing the standard
procedure of induction. \eop
\end{proof} 

According to the proof of  Theorem~\ref{bio:thm:main:1}, we have an
extension algorithm for Theorem~\ref{bio:thm:main:1}. See
Algorithm~1.
\begin{algorithm}[H]\label{bio:alg:1}
\caption{Biorthogonal Matrix Extension with Symmetrty}
\begin{algorithmic}[1]
\item[{\rm(a)}] {\bf Input}: $\pP,\wt\pP$ as in Theorem \ref{bio:thm:main:1} with
$\sym\pP=\sym\wt\pP=(\sym\pth_1)^*\sym\pth_2$ for two $1\times r$,
$1\times s$ row vectors $\pth_1$, $\pth_2$ of Laurant polynomials
with symmetry.

\item[{\rm(b)}] {\bf Initialization}:
 Let  $\pQ :=
\pU_{\sym\pth_1}^*\pP\pU_{\sym\pth_2}$ and  $\wt\pQ :=
\pU_{\sym\pth_1}^*\wt\pP\pU_{\sym\pth_2}$. Then both $\pQ$ and
$\wt\pQ$ have the the same symmetry pattern as follows:
\begin{equation}\label{bio:QQ:sym:standard}
\sym \pQ=\sym\wt\pQ=[ {\bf1}_{ r_1},-{\bf1}_{ r_2 }, z{\bf1}_{ r_3
}, -z{\bf1}_{ r_4} ]^T[{\bf1}_{ s_1},-{\bf1}_{ s_2},  z^{-1}{\bf1}_{
s_3},-z^{-1}{\bf1}_{ s_4}],
\end{equation}
where all nonnegative integers $r_1,\ldots, r_4, s_1, \ldots, s_4$
are uniquely determined by $\sym \pP$. Note that this step does not
increase the lengths of the  coefficient  support of both $\pP$ and
$\wt\pP$.

\item[{\rm(c)}]{\bf Support Reduction}:
\STATE Let $\pU_0:=\pU_{\sym\pth_2}^*$ and $\pA=\wt\pA:=I_{s}$.

\FOR{$k = 1$ to $r$}

\STATE Let $\pp:=[\pQ]_{k,k:s}$ and $\wt\pp:=[\wt\pQ]_{k,k:s}$.

\WHILE{$|\cs(\pp)|>0$ {\rm\texttt{and}} $|\cs(\wt\pp)|>0$} \STATE
Construct a pair of biorthogonal matrices
$\left(\pB(z),\wt\pB(z)\right)$ with respect to the pair $(\pp,
\wt\pp)$ by Lemma~\ref{lemma:bioDegBy1} such that
\[
|\cs(\pp\pB)|+|\cs(\wt\pp\wt\pB)|<|\cs(\pp)|+|\cs(\wt\pp)|.
\]

\STATE Replace $\pp,\wt\pp$ by $\pp\pB$, $\wt\pp\wt\pB$,
respectively. \STATE Set $\pA:=\pA\diag(I_{k-1},\pB)$ and
$\wt\pA:=\wt\pA \diag(I_{k-1},\wt\pB)$. \ENDWHILE

\STATE The pair $(\pp,\wt\pp)$ is  of the form:
$([1,0,\ldots,0],[1,\pq(z)])$ for some $1\times(s-k)$ vector of
Laurent polynomials $\pq(z)$. Construct $\pB(z),\wt\pB(z)$ as
follows:
\[
\pB(z)=\left[
               \begin{array}{cc}
                 1 & 0 \\
                 \pq^*(z) & I_{s-k} \\
               \end{array}
             \right], \wt\pB(z)=\left[
               \begin{array}{cc}
                 1 & -\pq(z) \\
                {\bf0} & I_{s-k} \\
               \end{array}
             \right].
\]
\STATE Set $\pA:=\pA\diag(I_{k-1},\pB)$ and $\wt\pA:=\wt\pA
\diag(I_{k-1},\wt\pB)$.

\STATE Set $\pQ:=\pQ\pA$ and $\wt\pQ:=\wt\pQ\wt\pA$.

 \ENDFOR

\item[{\rm(d)}] {\bf Finalization}:
Let $\pU_{1}:=\diag(\pU_{\sym\pth_1},I_{s-r})$. Set
$\pP_e:=\pU_1\pA^*\pU_0$ and $\wt\pP_e:=\pU_1\wt\pA^*\pU_0$.

\item[{\rm(e)}]{\bf Output}: A pair of  desired matrices $(\pP_e,\wt\pP_e)$
satisfying all the properties in Theorem~\ref{bio:thm:main:1}.
\end{algorithmic}
\end{algorithm}
\section{Application to Biorthogonal Multi\-wave\-lets with Symmetry}
In this section, we shall discuss the connection between  matrix
extension and  biorthogonal multiwavelets. We shall also discuss the
application of our results obtained in previous section to the
construction of biorthogonal multiwavelets with symmetry. Several
examples are provided to demonstrate our results.

We say that $\df$ is {\it a dilation factor} if $\df$ is an integer
with $|\df|>1$. Throughout this section, $\df$ denotes a dilation
factor. For simplicity of presentation, we further assume that $\df$
is positive, while multiwavelets and filter banks with a negative
dilation factor can be handled similarly by a slight modification of
the statements in this paper.

Let $\F$ be a subfield of $\CC$. A low-pass filter $\ta_0: \Z
\mapsto \F^{\mphi\times \mphi}$ with multiplicity $\mphi$ is a
finitely supported sequence of $\mphi\times \mphi$ matrices on $\Z$.
The \emph{symbol} of the filter $\ta_0$ is defined to be
$\pa_0(z):=\sum_{k\in \Z} \ta_0(k) z^k$, which is a matrix of
Laurent polynomials with coefficients in $\F$. Let $\df$ be a
dilation factor and $d_1, d_2$ be two fixed number in $\F$ such that
$\df=d_1d_2$ (for instance $d_1=1,d_2=2$ for $\df=2$ if $\F=\Q$).
Let $(a_0,\wt a_0)$ be a pair of low-pass filters with multiplicity
$r$. We say that $(a_0,\wt a_0)$ is a pair of \emph{biorthogonal
$\df$-band filters} if
\begin{equation}\label{mask:biorth}
\sum_{\gamma=0}^{\df-1} \pa_{0;\gamma}(z) \wt\pa_{0;
\gamma}^*(z)=I_\mphi, \qquad z\in \CC \bs \{0\},
\end{equation}
where $\pa_{0;\gamma}$ and $\wt\pa_{0;\gamma}$ are \emph{$\df$-band
subsymbols (polyphases, cosets)}  of $\pa_0$ and $\wt\pa_0$ defined
to be
\begin{equation}\label{bio:def:polyphases}
\begin{array}{l}
\pa_{0;\gamma}(z):=d_1\sum_{k\in\Z}a_0(k+\df k)z^k,\\
\wt\pa_{0;\gamma}(z):=d_2\sum_{k\in\Z}\wt a_0(k+\df k)z^k,
\end{array}
\quad \gamma\in\Z.
\end{equation}
%

Quite often, a low-pass filter $a_0$ is obtained beforehand. To
construction a pair of biorthogonal $\df$-band filters $(a_0,\wt
a_0)$, i.e., \eqref{mask:biorth} holds, one can use the CBC
(\emph{coset-by-coset}) algorithm proposed in \cite{Han:2001:JAT} to
derive $\wt a_0$ from $a_0$. There are two key ingredients in the
CBC algorithm. One is that the CBC algorithm reduces the nonlinear
system in the definition of sum rules for $\wt a_0$ to a system of
linear equations. Another is that the CBC algoirithm reduces the big
system of linear equation of biorthogonality relation for the pair
$(a_0,\wt a_0)$ to a small systems of linear equations in
\eqref{mask:biorth}. Moreover, the CBC algorithm guarantees that for
any given positive integers $\wt \kappa$, there always exists a
finitely supported filter $\wt a_0$ that satisfies the sum rules of
order $\wt\kappa$. For more details on the CBC algorithm, one may
refer to \cite{Han:2001:JAT,Han.Kwon.Zhuang:2009}. In our example
presented in this section, the pairs of biorthogonal $\df$-band
low-pass filters are obtained via this way using the CBC algorithm
(see examples in \cite{Han.Kwon.Zhuang:2009}).

For $f\in L_1(\R)$, the Fourier transform used  is defined to be
$\hat f(\xi):=\int_\R f(x) e^{-i x\xi} dx$ and can be naturally
extended to $L_2(\R)$ functions. For a pair of biorthogonal
$\df$-band filter $(\pa_0,\wt\pa_0)$, we assume that there exist
{\it a pair of biorthogonal $\df$-refinable function vectors}
$(\phi,\wt\phi)$ associated with the pair of biorthogonal $\df$-band
filters $(\pa_0,\wt\pa_0)$. That is,
\begin{equation}\label{F:refeq}
\wh \phi(\df \xi)=\pa_0(e^{-i\xi}) \wh \phi(\xi), \quad \wh
{\wt\phi}(\df \xi)=\wt\pa_0(e^{-i\xi}) \wh{\wt \phi}(\xi)\qquad
\xi\in \R,
\end{equation}
and
\begin{equation}\label{orth:reffunc}
\la \phi(\cdot -k), \wt\phi\ra:=\int_\R \phi(x-k) \ol{\wt\phi(x)}^T
dx=\gd(k) I_\mphi, \qquad k\in \Z,
\end{equation}
where $\gd$ denotes the {\it Dirac sequence} such that $\gd(0)=1$
and $\gd(k)=0$ for all $k\ne 0$.

To construct biorthogonal multiwavelets in $L_2(\R)$, we need to
design high-pass filters $\ta_1, \ldots, \ta_{\df-1}: \Z \rightarrow
\F^{\mphi\times \mphi}$ and $\wt\ta_1,\ldots,\wt\ta_{\df-1}: \Z
\rightarrow \F^{\mphi\times \mphi}$ such that the polyphase matrices
with respect to the filter banks
$\{\pa_0,\pa_1,\ldots,\pa_{\df-1}\}$ and
$\{\wt\pa_0,\wt\pa_1,\ldots,\wt\pa_{\df-1}\}$
\begin{equation}\label{bio:polyphaseMatrix}
\begin{small}
\PP(z)=\left[ \begin{matrix} \pa_{0;0}(z) &\cdots &\pa_{0; \df-1}(z)\\
\pa_{1;0}(z) &\cdots &\pa_{1; \df-1}(z)\\
\vdots &\vdots &\vdots\\
\pa_{\df-1;0}(z) &\cdots &\pa_{\df-1; \df-1}(z)
\end{matrix}\right], \,\wt\PP(z)=\left[ \begin{matrix} \wt\pa_{0;0}(z) &\cdots &\wt\pa_{0; \df-1}(z)\\
\wt\pa_{1;0}(z) &\cdots &\wt\pa_{1; \df-1}(z)\\
\vdots &\vdots &\vdots\\
\wt\pa_{\df-1;0}(z) &\cdots &\wt\pa_{\df-1; \df-1}(z)
\end{matrix}\right]
\end{small}
\end{equation}
are biorthogonal, that is, $\PP(z) \wt\PP^*(z)=I_{\df\mphi}$, where
 $\pa_{m;\gamma}, \wt\pa_{m;\gamma}$ are
 subsymbols of $\pa_m,\wt\pa_m$ defined similar to \eqref{bio:def:polyphases} for
$m,\gamma=0,\ldots,\df-1$, respectively. The pair of filter banks
$(\{\pa_0,\ldots,\pa_{\df-1}\},\{\wt \pa_0,\ldots,\wt
\pa_{\df-1}\})$ satisfying $\PP\wt\PP^*=I_{\df r}$ is called \emph{a
pair of biorthogonal filter banks with the perfect reconstruction
property}.

Symmetry of the filters in a filter bank is a very much desirable
property in many applications. We say that the low-pass filter
$\pa_0$ (or $\ta_0$) has symmetry if
\begin{equation}\label{mask:sym}
\pa_0(z)=\mbox{diag}(\gep_1 z^{\df c_1}, \ldots, \gep_{\mphi} z^{\df
c_{\mphi}}) \pa_0(1/z) \mbox{diag}(\gep_1 z^{-c_1}, \ldots,
\gep_{\mphi} z^{-c_{\mphi}})
\end{equation}
for some $\gep_1, \ldots, \gep_{\mphi}\in \{-1,1\}$ and $c_1,
\ldots, c_{\mphi}\in \R$ such that $\df c_\ell-c_j\in\Z$ for all
$\ell,j=1,\ldots,r$. If $\pa_0$ has symmetry as in \eqref{mask:sym}
and if $1$ is a simple eigenvalue of $\pa_0(1)$, then it is well
known that the $\df$-refinable function vector $\phi$ in
\eqref{F:refeq} associated with the low-pass filter $\pa_0$ has the
following symmetry:
\begin{equation}\label{phi:sym}
\phi_1(c_1-\cdot)=\gep_1 \phi_1,\quad \phi_2(c_2-\cdot)=\gep_2
\phi_2,\quad \ldots, \quad \phi_\mphi(c_\mphi-\cdot)=\gep_\mphi
\phi_\mphi.
\end{equation}
%


Under the symmetry condition in \eqref{mask:sym}, to apply Theorem
\ref{bio:thm:main:1}, we first show that there exists  a suitable
\emph{invertible} matrix $\pU$, i.e., $\det(\pU)$ is a monomial, of
Laurent polynomials in $\F$ acting on
$\pP_{\pa_0}:=[\pa_{0;0},\ldots,\pa_{0;\df-1}]$ so that
$\pP_{\pa_0}\pU$ has compatible symmetry. Note that $\pP_{\pa_0}$
itself may not have compatible symmetry.

\begin{lemma}\label{lemma:symPa}
Let $\pP_{\pa_0}:=[\pa_{0;0},\ldots,\pa_{0;\df-1}]$, where
$\pa_{0;0}, \ldots, \pa_{0; \df-1}$ are $\df$-band subsymbols of a
low-pass filter $\pa_0$ satisfying \eqref{mask:sym}. Then there
exists a $\df r\times \df r$ invertible matrix $\pU$ of Laurent
polynomials with symmetry such that $\pP_{\pa_0}\pU$ has compatible
symmetry.
\end{lemma}

\begin{proof}
From \eqref{mask:sym}, we deduce that
%
\begin{equation}\label{eq:polyPhaseSymCondition}
[\pa_{0;\gamma}(z)]_{\ell,j}=\varepsilon_\ell\varepsilon_jz^{R_{\ell,j}^\gamma}[\pa_{0;{Q_{\ell,j}^\gamma}}(z^{-1})]_{\ell,j},\,
\gamma=0,\ldots,\df-1; \ell,j=1,\ldots,r,
\end{equation}
%
where $\gamma,Q_{\ell,j}^\gamma\in \Gamma:=\{0,\ldots,\df-1\}$ and
$R_{\ell,j}^\gamma$, $Q_{\ell,j}^\gamma$ are uniquely determined by
\begin{equation}\label{RQ}
\df c_\ell-c_j-\gamma=\df R_{\ell,j}^\gamma+Q_{\ell,j}^\gamma \quad
\mbox{with} \quad R_{\ell,j}^\gamma\in \Z, \;
Q_{\ell,j}^\gamma\in\Gamma.
\end{equation}
Since $\df c_\ell-c_j\in\Z$ for all $\ell, j=1,\ldots, r$, we have
$c_\ell-c_j\in\Z$ for all $\ell,j=1,\ldots,r$ and therefore,
$Q_{\ell,j}^\gamma$ is independent of $\ell$. Consequently, by
\eqref{eq:polyPhaseSymCondition}, for every $1\le j\le r$, the $j$th
column of the matrix $\pa_{0;\gamma}$ is a flipped version of the
$j$th column of the matrix $\pa_{0;{Q_{\ell,j}^\gamma}}$.
Let $\kappa_{j,\gamma}\in\Z$ be an integer such that
$|\cs([\pa_{0;\gamma}]_{:,j}
+z^{\kappa_{j,\gamma}}[\pa_{0;{Q_{\ell,j}^\gamma}}]_{:,j})|$
is as small as possible. Define
$\pP:=[\pb_{0;0},\ldots,\pb_{0;\df-1}]$ as follows:
\begin{equation}\label{eq:symmetrization}
[\pb_{0;\gamma}]_{:,j}:= \begin{cases}
      [\pa_{0;\gamma}]_{:,j}, & \hbox{$\gamma=Q_{\ell,j}^\gamma;$} \\
       \frac{1}{2}([\pa_{0;\gamma}]_{:,j}+ z^{\kappa_{j,\gamma}}[\pa_{0;{Q_{\ell,j}^\gamma}}]_{:,j}), & \hbox{$\gamma<Q_{\ell,j}^\gamma$;}\\
       \frac{1}{2}([\pa_{0;\gamma}]_{:,j}-z^{\kappa_{j,\gamma}}[\pa_{0;{Q_{\ell,j}^\gamma}}]_{:,j}), & \hbox{$\gamma>Q_{\ell,j}^\gamma$,}
       \end{cases}
\end{equation}
where $[\pa_{0;\gamma}]_{:,j}$ denotes the $j$th column of
$\pa_{0;\gamma}$. Let $\pU$ denote the unique transform matrix
corresponding to \eqref{eq:symmetrization} such that
$\pP:=[\pb_{0;0},\ldots,\pb_{0;\df-1}]=[\pa_{0;0},\ldots,\pa_{0;\df-1}]
\pU$. It is evident that $\pU$ is paraunitary and
$\pP=\pP_{\pa_0}\pU$. We now show that $\pP$ has compatible
symmetry. Indeed, by \eqref{eq:polyPhaseSymCondition} and
\eqref{eq:symmetrization},
\begin{equation}\label{eq:symPatternPb}
[\sym \pb_{0;\gamma}]_{\ell,j}=
\sgn(Q_{\ell,j}^\gamma-\gamma)\gep_\ell\gep_jz^{R_{\ell,j}^\gamma+\kappa_{j,\gamma}},
\end{equation}
where $\sgn(x)=1$ for $x\ge 0$ and $\sgn(x)=-1$ for $x<0$.
By \eqref{RQ} and noting that $Q_{\ell,j}^\gamma$ is independent of
$\ell$, we have
\[
\frac{[\sym \pb_{0;\gamma}]_{\ell,j}}{[\sym \pb_{0;\gamma}]_{n,j}}
=\gep_\ell\gep_nz^{R_{\ell,j}^\gamma-R_{n,j}^\gamma}=
\gep_\ell\gep_nz^{c_\ell-c_n},
\]
for all $1\le \ell,n\le r$,
which is equivalent to saying that $\pP$ has compatible symmetry.
\eop
\end{proof}

Now, for a pair of  biorthogonal $\df$-band  low-pass filters
$(\pa_0,\wt\pa_0)$ with multiplicity $r$ satisfying
\eqref{mask:sym}, we have an algorithm (see Algorithm~2) to
construct high-pass filters $\pa_1,\ldots,\pa_{\df-1}$ and
$\wt\pa_1,\ldots,\wt\pa_{\df-1}$ such that the polyphase matrices
$\PP(z)$ and $\wt\PP(z)$ defined as in \eqref{bio:polyphaseMatrix}
satisfy $\PP(z)\wt\PP^*(z)=I_{\df r}$. Here,
$\pP_{\pa_0}:=[\pa_{0;0},\ldots,\pa_{0;\df-1}]$ and
$\wt\pP_{\wt\pa_0}:=[\wt\pa_{0;0},\ldots,\wt\pa_{0;\df-1}]$ are the
polyphase vectors of $\pa_0,\wt\pa_0$ obtained by
\eqref{bio:def:polyphases}, respectively.

\begin{algorithm}\label{bio:alg:2}
\caption{Construction of Biorthogonal Multiwavelets with Symmetry}
\begin{algorithmic}[1]
\item[{\rm(a)}] {\bf Input}: $(\pa_0,\wt\pa_0)$, a pair of biorthogonal $\df$-band
filters with multiplicity $r$ and with the same symmetry as in
\eqref{mask:sym}.

\item[{\rm(b)}] {\bf Initialization}:
Construct a pair of biorthogonal matrices $(\pU,\wt\pU)$ in $\F$ by
Lemma~\ref{lemma:symPa} such that both $\pP:=\pP_{\pa_0}\pU$ and
$\wt\pP=\wt\pP_{\wt\pa_0}\wt\pU$ ($\wt\pU=(\pU^*)^{-1})$ are
matrices of Laurent polynomials with coefficient in $\F$  having
compatible symmetry:
$\sym\pP=\sym\wt\pP=[\gep_1z^{k_1},\ldots,\gep_r z^{k_r}]^T\sym\pth$
for some $k_1,\ldots,k_r\in\Z$ and some $1\times \df r$ row vector
$\pth$ of Laurent polynomials with symmetry.

\item[{\rm(c)}]{\bf Extension}: Derive $\pP_e, \wt\pP_e$ with all the properties as in
Theorem~\ref{bio:thm:main:1} from $\pP,\wt\pP$ by
Algorithm~\ref{bio:alg:1}.

\item[{\rm(d)}]{\bf High-pass Filters}: Let $\PP:=\pP_e\wt\pU^*=:(\pa_{m;\gamma})_{0\le m,\gamma\le \df-1}$,
$\wt\PP:=\wt\pP_e\pU^*=:(\wt\pa_{m;\gamma})_{0\le m,\gamma\le
\df-1}$
 as  in \eqref{bio:polyphaseMatrix}. For $m=1,\ldots,\df-1$, define high-pass filters
\begin{equation}
\label{highpass:def} \pa_m(z):=\frac{1}{d_1}
\sum_{\gamma=0}^{\df-1}\pa_{m;\gamma}(z^\df)z^\gamma,\quad
\wt\pa_m(z):=\frac{1}{d_2}
\sum_{\gamma=0}^{\df-1}\wt\pa_{m;\gamma}(z^\df)z^\gamma.
\end{equation}

\item[{\rm(e)}]{\bf Output}: a pair of biorthogonal filter banks $(\{ \pa_0, \pa_1, \ldots,
\pa_{\df-1}\},\{ \wt\pa_0, \wt\pa_1, \ldots, \wt\pa_{\df-1}\})$ with
symmetry and with the perfect reconstruction property, i.e. $\PP,
\wt\PP$ in \eqref{bio:polyphaseMatrix} are biorthogonal and all
filters $\pa_m, \wt\pa_m$, $m=1, \ldots, \df-1$, have symmetry:
\begin{equation}\label{bio:highpass:sym}
\begin{array}{c}
\pa_{m}(z)=\diag(\gep^m_1z^{\df c^m_1},\ldots,\gep^m_rz^{\df
c^m_r})\pa_m(1/z)\diag(\gep_1z^{-c_1},\ldots,\gep_{r}z^{-c_r}),
\\
\wt\pa_{m}(z)=\diag(\gep^m_1z^{\df c^m_1},\ldots,\gep^m_rz^{\df
c^m_r})\wt\pa_m(1/z)\diag(\gep_1z^{-c_1},\ldots,\gep_{r}z^{-c_r}),
\end{array}
\end{equation}
where $c^m_\ell:=(k^m_{\ell}-k_\ell)+c_{\ell}\in\R$ and all
$\varepsilon^m_{\ell}\in\{-1,1\}$, $k^m_{\ell}\in\Z$, for
$\ell=1,\ldots,r$ and  $m=1,\ldots,\df-1$, are determined by the
symmetry pattern of $\pP_e$ as follows:
\begin{equation}
\label{bio:sym:Pe} [\gep_1 z^{k_1},\ldots,\gep_r z^{k_r}, \gep^1_1
z^{k^1_1},\ldots,\gep^1_r
z^{k^1_r},\ldots,z^{k^{\df-1}_1},\ldots,\gep^{\df-1}_r
z^{k^{\df-1}_r}]^T\sym\pth:=\sym\pP_e.
\end{equation}

\end{algorithmic}
\end{algorithm}

Let $(\phi,\wt\phi)$ be a pair of biorthogonal $\df$-refinable
function vectors in $L_2(\R)$ associated with a pair of biorthogonal
 $\df$-band filters  $(\pa_0,\wt\pa_0)$ and with $\phi=[\phi_1, \ldots,
\phi_\mphi]^T$, $\wt\phi=[\wt\phi_1, \ldots, \wt\phi_\mphi]^T$.
Define multiwavelet function vectors
$\psi^m=[\psi^m_1,\ldots,\psi^m_r]^T$,
 $\wt\psi^m=[\wt\psi^m_1,\ldots,\wt\psi^m_r]^T$
associated with the high-pass filters $\pa_m, \wt\pa_m$,
$m=1,\ldots,\df-1$, by
\begin{equation}\label{bio:wavelet}
\wh{\psi^m}(\df\xi):=\pa_m(e^{-i\xi}) \wh{\phi}(\xi),\;
\wh{\wt\psi^m}(\df\xi):=\wt\pa_m(e^{-i\xi}) \wh{\wt\phi}(\xi),\,\;
\xi\in \R.
\end{equation}
It is well known that $\{\psi^1, \ldots, \psi^{\df-1}; \wt\psi^1,
\ldots, \wt\psi^{\df-1}\}$ generates a biorthonormal multiwavelet
basis in $L_2(\R)$.

Since the high-pass filters $\pa_1,\ldots,\pa_{\df-1}$,
$\wt\pa_1,\ldots,\wt\pa_{\df-1}$ satisfy \eqref{bio:highpass:sym},
it is easy to verify that each $\psi^m=[\psi^m_1, \ldots,
\psi^m_r]^T$, $\wt\psi^m=[\wt\psi^m_1, \ldots, \wt\psi^m_r]^T$
defined in \eqref{bio:wavelet} also has the following symmetry:
\begin{equation}\label{sym:psi}
\begin{array}{c}
\psi^m_{1}(c^m_{1}-\cdot)=\varepsilon^m_{1}\psi^m_{1},\quad
\psi^m_{2}(c^m_{2}-\cdot)=\varepsilon^m_{2}\psi^m_{2},\quad
\ldots,\quad\psi^m_{r}(c^m_{r}-\cdot)=\varepsilon^m_{r}\psi^m_{r},\\
\wt\psi^m_{1}(c^m_{1}-\cdot)=\varepsilon^m_{1}\wt\psi^m_{1},\quad
\wt\psi^m_{2}(c^m_{2}-\cdot)=\varepsilon^m_{2}\wt\psi^m_{2},\quad
\ldots,\quad
\wt\psi^m_{r}(c^m_{r}-\cdot)=\varepsilon^m_{r}\wt\psi^m_{r}.
\end{array}
\end{equation}

In the following, let us present several examples to demonstrate our
results and illustrate our algorithms.

\begin{example} \label{bio:ex:2}  {\rm
Let $\df=r=2$ 
and $a_0, \wt a_0$ be a pair of dual $\df$-filters with symbols
$\pa_0(z), \wt\pa_0(z)$ (cf. \cite{Han.Kwon.Zhuang:2009}) given by
\[
\begin{aligned}
\pa_0(z)&=\frac{1}{16}\left[ \begin {array}{cc} 8&6\,{z}^{-1}+6\\
\noalign{\medskip}8\,z&-{z}^{-1}+3+3\,z-{z}^{2}\end {array} \right]
,\\
 \wt\pa_0(z)&=\frac{1}{384} \left[ \begin {array}{cc}
-28\,{z}^{-1}+216-28\,z&112\,{z}^{-1}+112
\\ \noalign{\medskip}21\,{z}^{-1}-18+330\,z-18\,{z}^{2}+21\,{z}^{3}&-
36\,{z}^{-1}+60+60\,z-36\,{z}^{2}\end {array} \right].
\end{aligned}
\]
Both $\pa_0$ and $\wt\pa_0$ have the same symmetry pattern and
satisfy \eqref{mask:sym}. Let $\df=d_1d_2$ with $d_1=1$  and
$d_2=2$. Then,
 $\pP_{\pa_0}:=[\pa_{0;0}, \pa_{0;1}]$ and
$\pP_{\pa_0}:=[\pa_{0;0}, \pa_{0;1}]$ are as follows:
\[
\begin{aligned}
\pP_{\pa_0}&=\frac{1}{16}\left[ \begin {array}{cccc}
8&6&0&6\,{z}^{-1}\\ \noalign{\medskip}0&3-z&8&-z^{-1}+3\end {array}
\right],
\\ \wt \pP_{\wt\pa_0}&=
\frac{1}{192}
\left[
\begin {array}{cccc}
216&112&-28(z^{-1}+1)&112{z}^{-1}\\
\noalign{\medskip} -18(1+z)&12(5-3z)&3(7z^{-1}+110+7z)&12(5-3z^{-1})
\end {array}
\right].
\end{aligned}
\]
Let $\pU$ and $\wt\pU$ be defined by
\[
\pU:=\left[
       \begin{array}{cccc}
         1 & 0 & 0 & 0 \\
         0 & 1 & 0 &1 \\
         0 & 0 & 1 & 0 \\
         0 & z & 0 & -z \\
       \end{array}
     \right]
, \wt\pU:=\frac12\left[
       \begin{array}{cccc}
         2 & 0 & 0 & 0 \\
         0 & 1 & 0 &1 \\
         0 & 0 & 2 & 0 \\
         0 & z & 0 & -z \\
       \end{array}
     \right].
\]
Then we have $\pU\wt\pU^*=I_4$. Let $\pP:=\pP_{\pa_0}\pU$ and
$\wt\pP:=\wt\pP_{\wt\pa_0}\pU$. Then we have
$\sym\pP=\sym\wt\pP=[1,z]^T[1,1,z^{-1},-1]$ and $\pP,\wt\pP$ are
given as follows:
\[
\begin{aligned}
\pP&=\frac{1}{8}\left[ \begin {array}{cccc} 4&6&0&0\\
\noalign{\medskip}0&1(1+z)&4&2(1-z)\end {array}
\right],\\
\wt\pP&=\frac{1}{192}\left[ \begin {array}{cccc} 216&112&-28(1+z^{-1})&0\\
\noalign{\medskip} -18(1+z)&12(1+z)&3(7z^{-1}+110+7z)& 48(1-z)\end
{array} \right].
\end{aligned}
\]
Now applying Algorithm~2, we obtain two extension matrices $\pP_e$
and $\wt\pP_e$ as follows:
\[
\begin{aligned}
\pP_e&=\frac{1}{192}\left[ \begin {array}{cccc} 96&144&0&0\\
\noalign{\medskip}0&24(1+z)&96&48(1-z)\\
\noalign{\medskip}-{112}&-3(z^{-1}-70+z)&-12(1+{z}^{-1})&-6(z^{-1}-z)\\
\noalign{\medskip}0&-6(z-z^{-1})&-24(1-{z}^{-1})&12(z+14+{z}^{-1})\end
{array} \right]
, \\
\wt\pP_e& = \frac{1}{192}\left[ \begin {array}{cccc} 216&112&-28(1+z^{-1})&0\\
\noalign{\medskip} -18(1+z)&12(1+z)&3(7z^{-1}+110+7z)& 48(1-z)\\
\noalign{\medskip}-144&96&-24(1+{z}^{-1})&0\\
\noalign{\medskip}0&0&-96(1-{z}^{-1})&192
\end {array} \right]
.
\end{aligned}
\]
Note that $\sym\pP_e=\sym\wt\pP_e=[1,z,1,-1]^T[1,1,z^{-1},-1]$. Now
from the polyphase matrices
$\PP:=\pP_e\wt\pU^*=:(\pa_{m;\gamma})_{0\le m,\gamma\le1}$ and
$\wt\PP:=\wt\pP_e\pU^*=:(\wt\pa_{m;\gamma})_{0\le m,\gamma\le1}$, we
derive two high-pass filters $\pa_1, \wt\pa_1$ as follows:
\[
\begin{aligned}
\pa_1(z)&=\frac{1}{384}\left[ \begin {array}{cc} -8(3z+28+3{z}^{-1})&3({z}^{2}-3z+70 +70z^{-1}-3z^{-2}+z^{-3})\\
\noalign{\medskip}-48(z-{z}^{-1})&6({z
}^{2}-3z+28-28{z}^{-1}+3{z}^{-2}-{z}^{-3})\end {array}
 \right],\\
 \wt\pa_1(z)&=\frac{1}{16}
 \left[ \begin {array}{cc} -(z+6+z^{-1})&4(1+z^{-1})\\
 \noalign{\medskip}-4(z-
{z}^{-1})&8(1-{z}^{-1})\end {array} \right].
\end{aligned}
\]
See Figure~3.1 for the graphs of $\phi=[\phi_1,\phi_2]^T$,
$\psi=[\psi_1,\psi_2]^T$, $\wt\phi=[\wt\phi_1,\wt\phi_2]^T$, and
$\wt\psi=[\wt\psi_1,\wt\psi_2]^T$.

\begin{figure}[th] \label{fig:ex1}
\centerline{\scalebox{0.75}{
\hbox{\epsfig{file=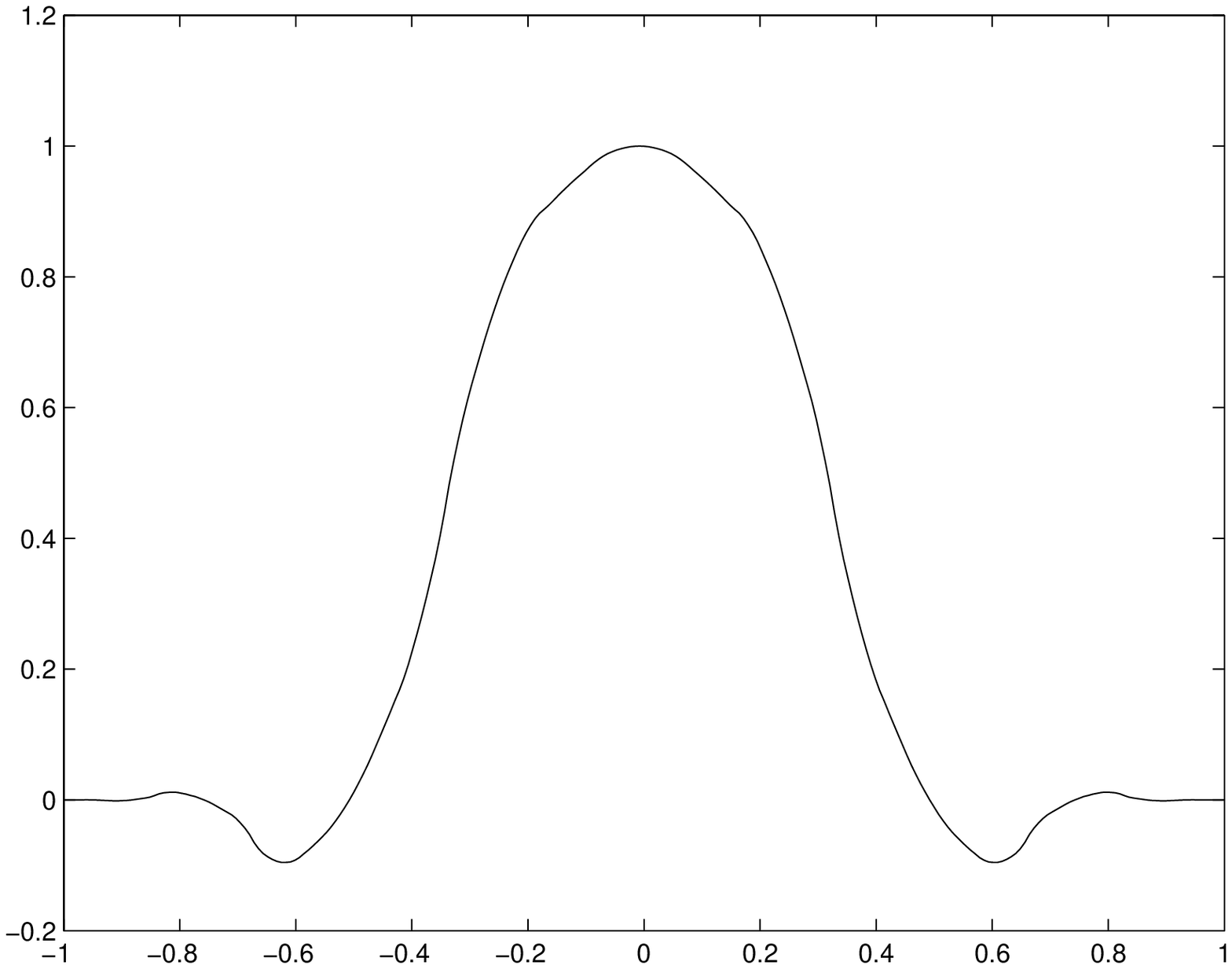,width=1.5in,height=1.5in}
\epsfig{file=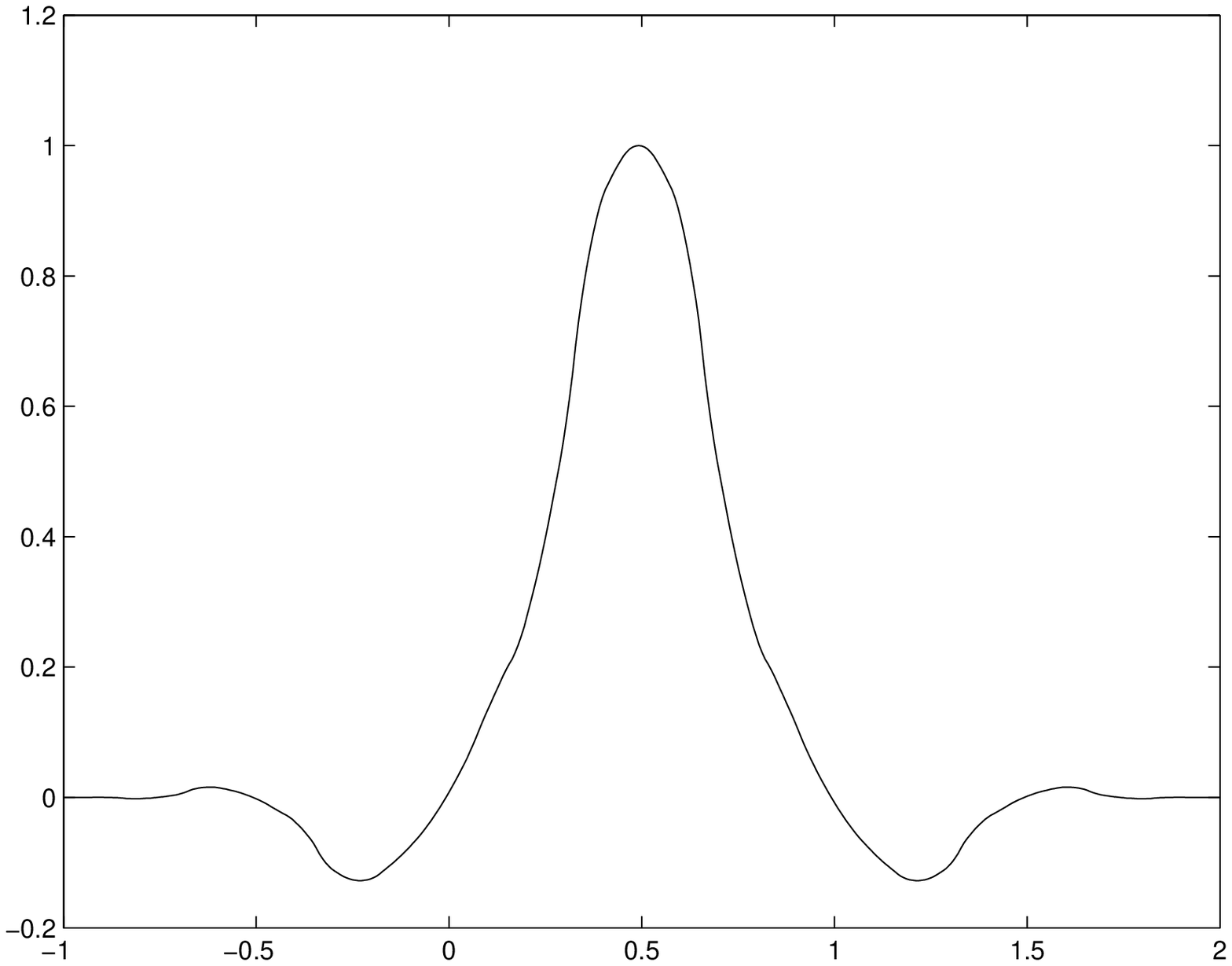,width=1.5in,height=1.5in}
\epsfig{file=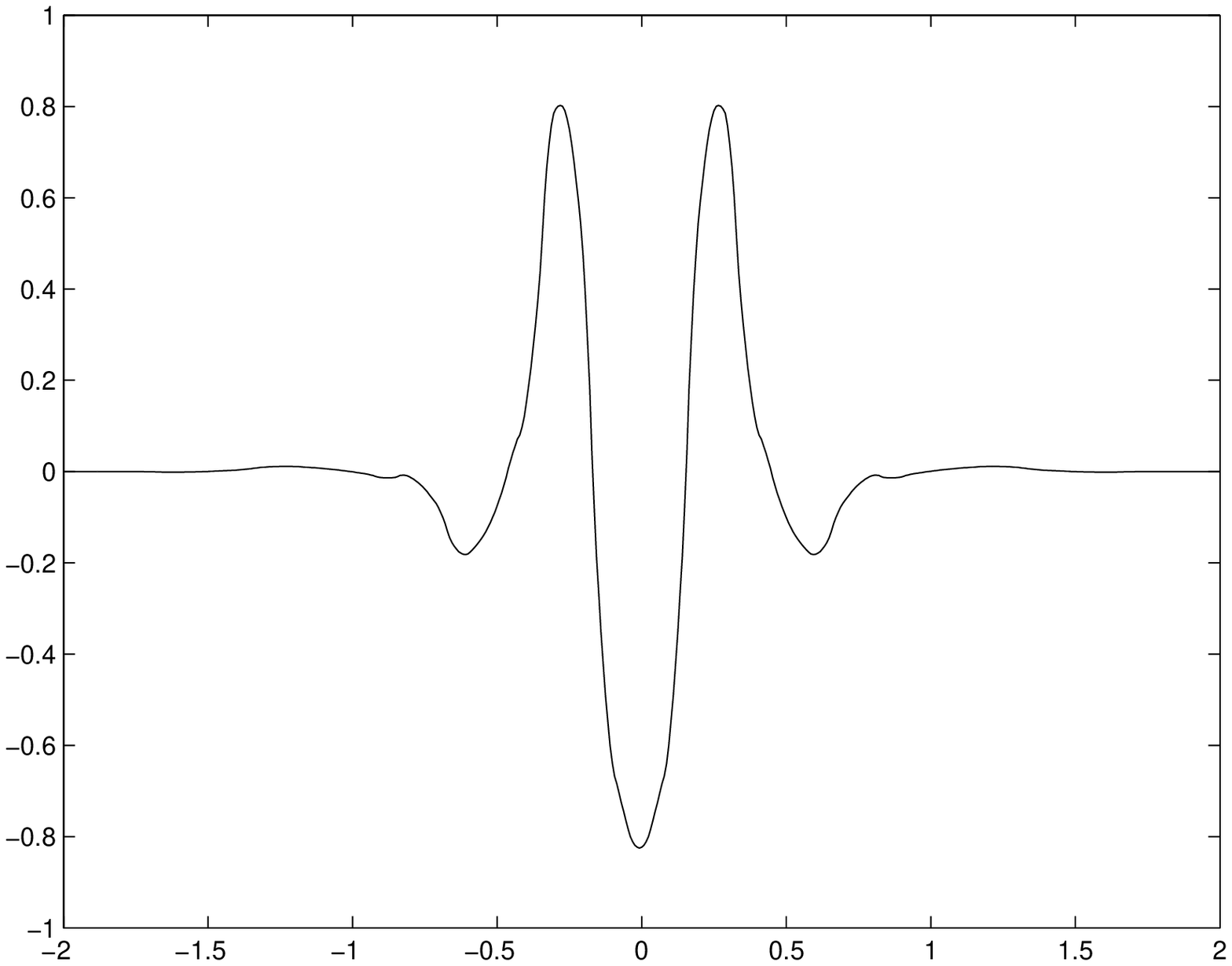,width=1.5in,height=1.5in}
\epsfig{file=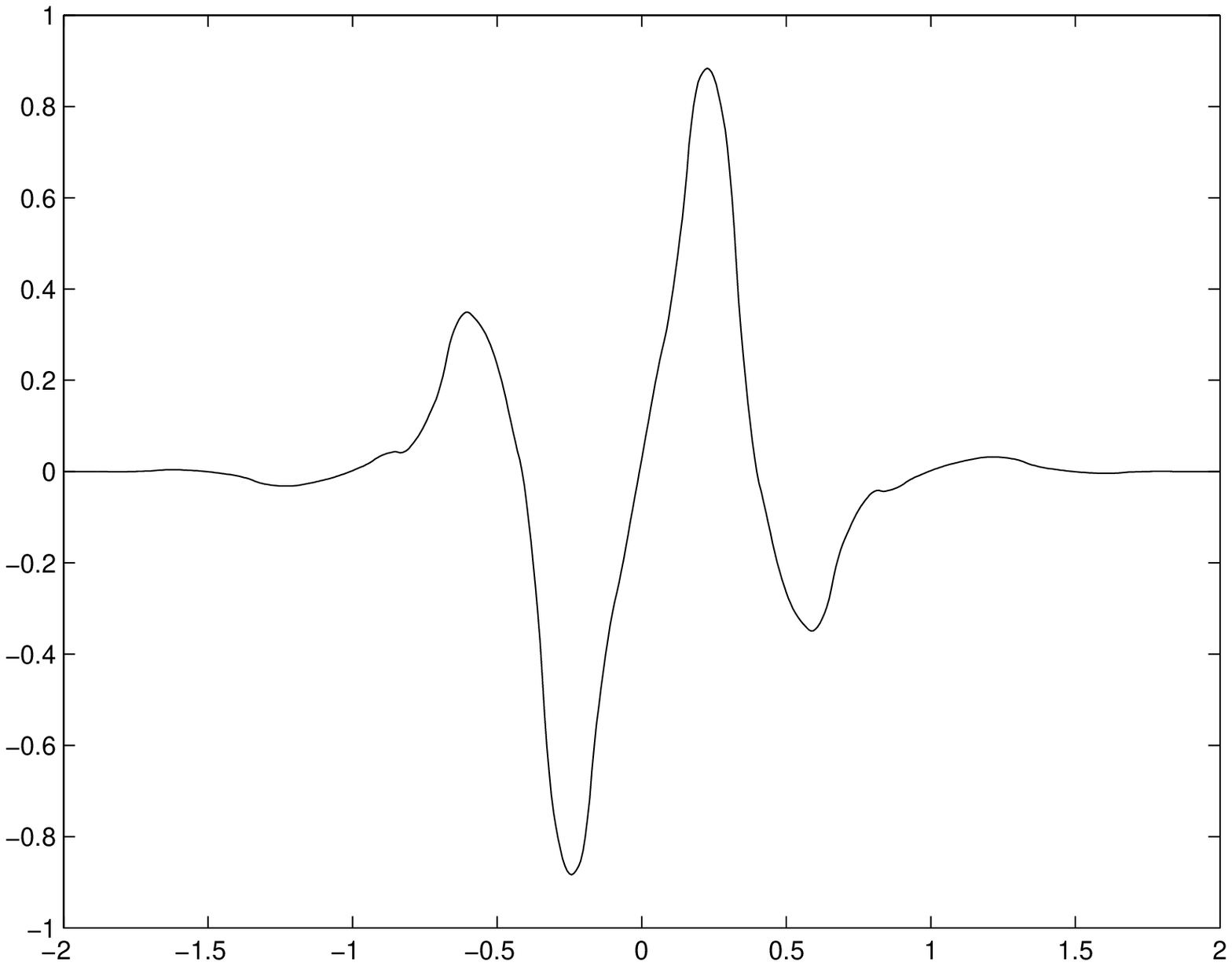,width=1.5in,height=1.5in} }}}
\centerline{\scalebox{0.75}{
\hbox{\epsfig{file=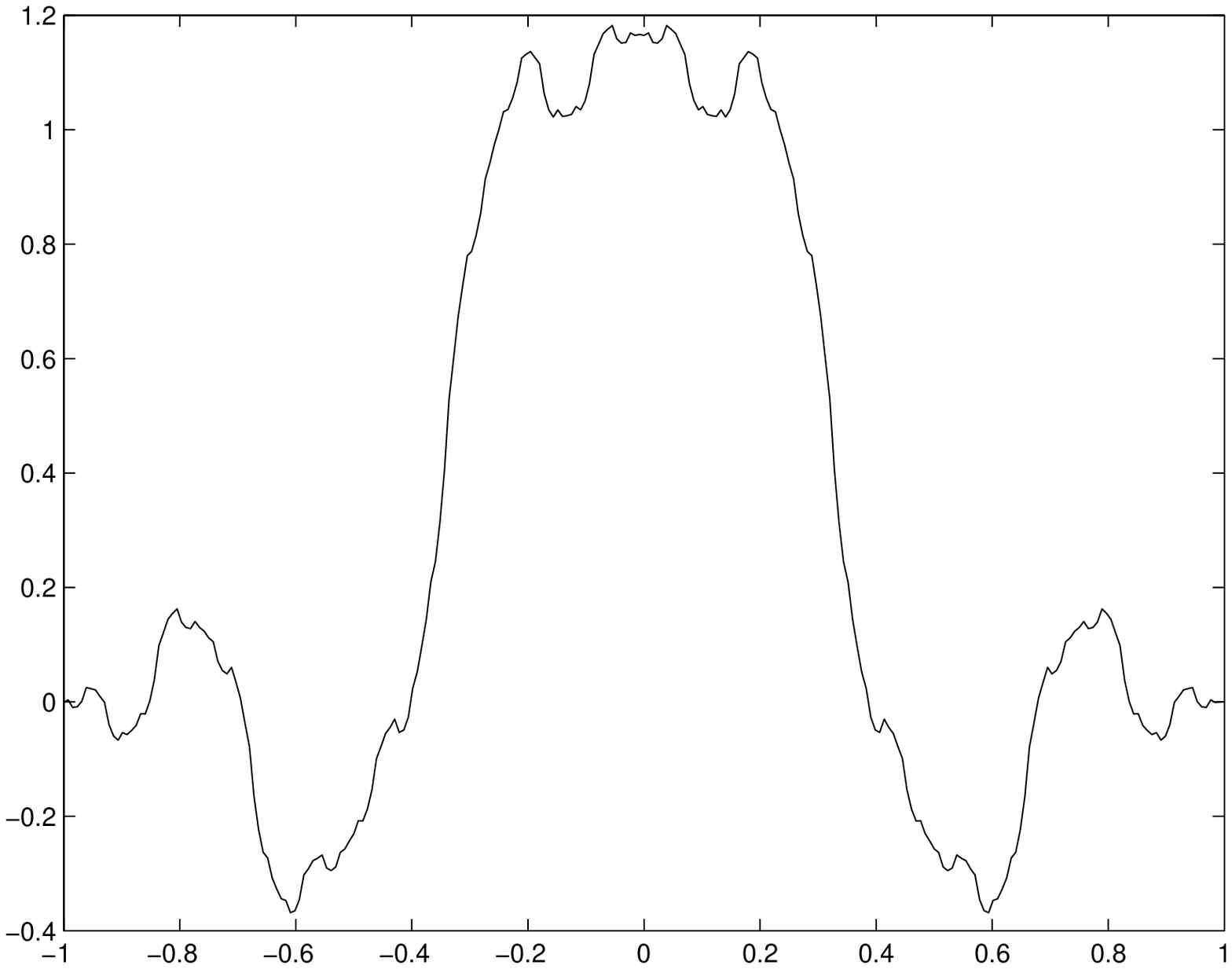,width=1.5in,height=1.5in}
\epsfig{file=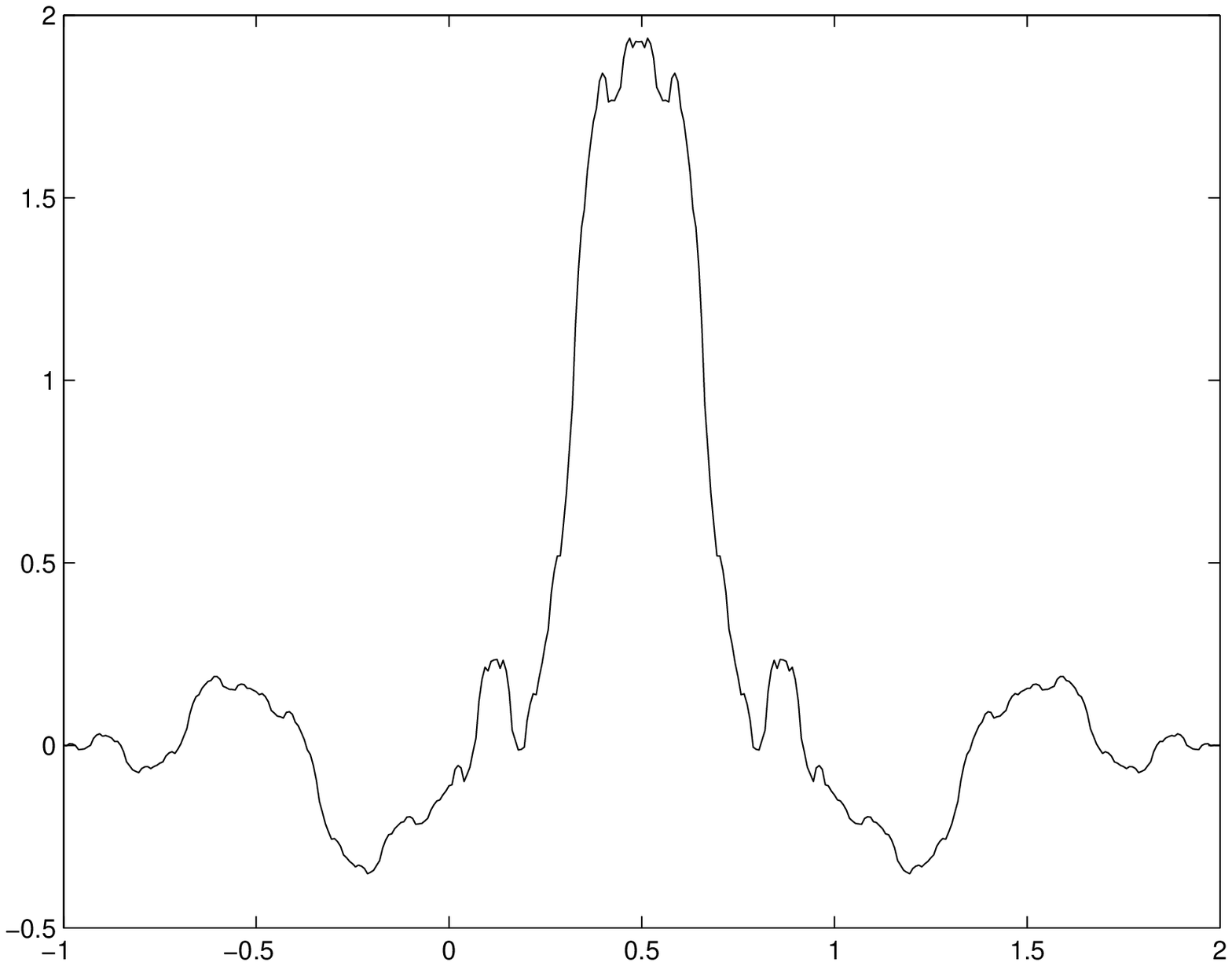,width=1.5in,height=1.5in}
\epsfig{file=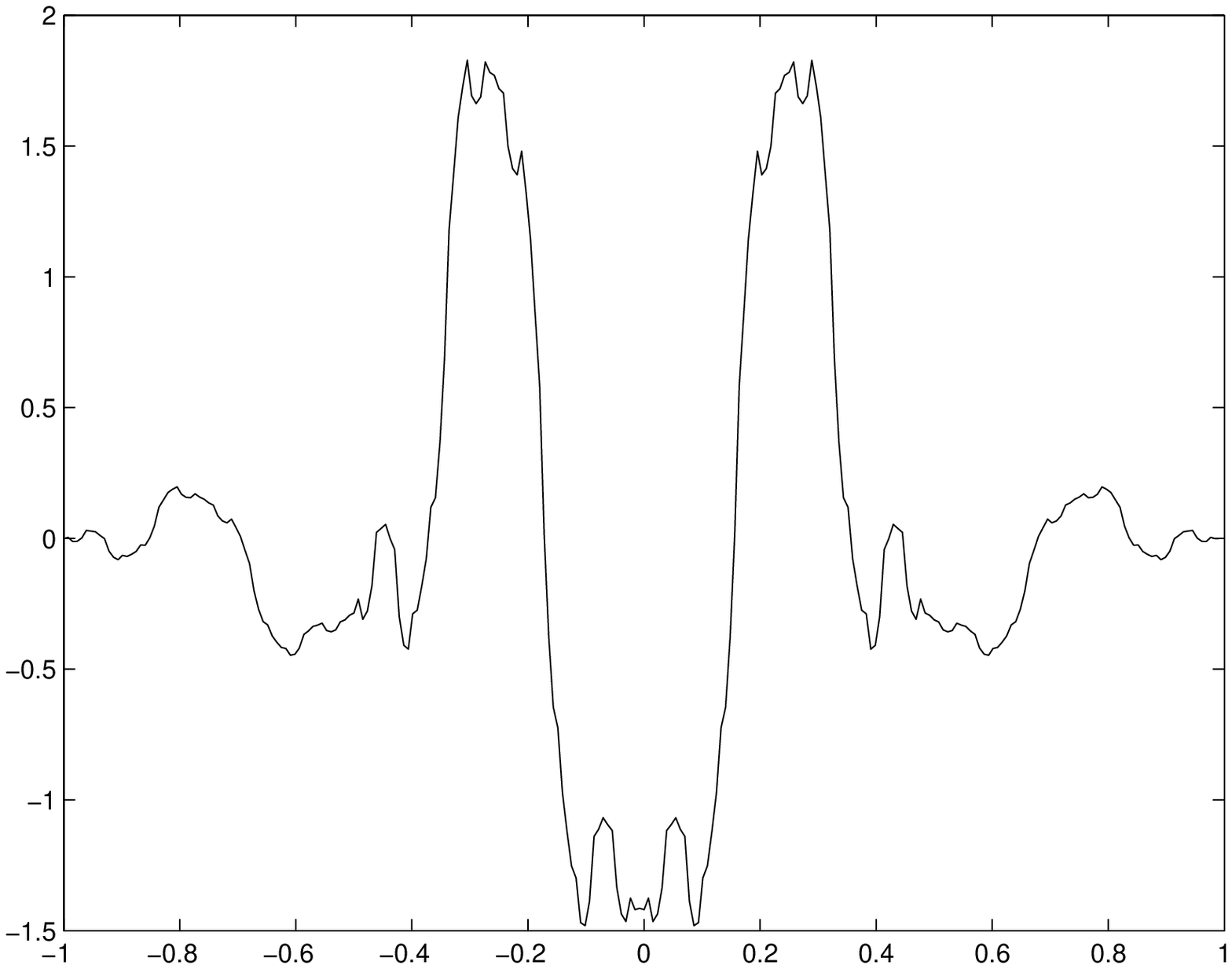,width=1.5in,height=1.5in}
\epsfig{file=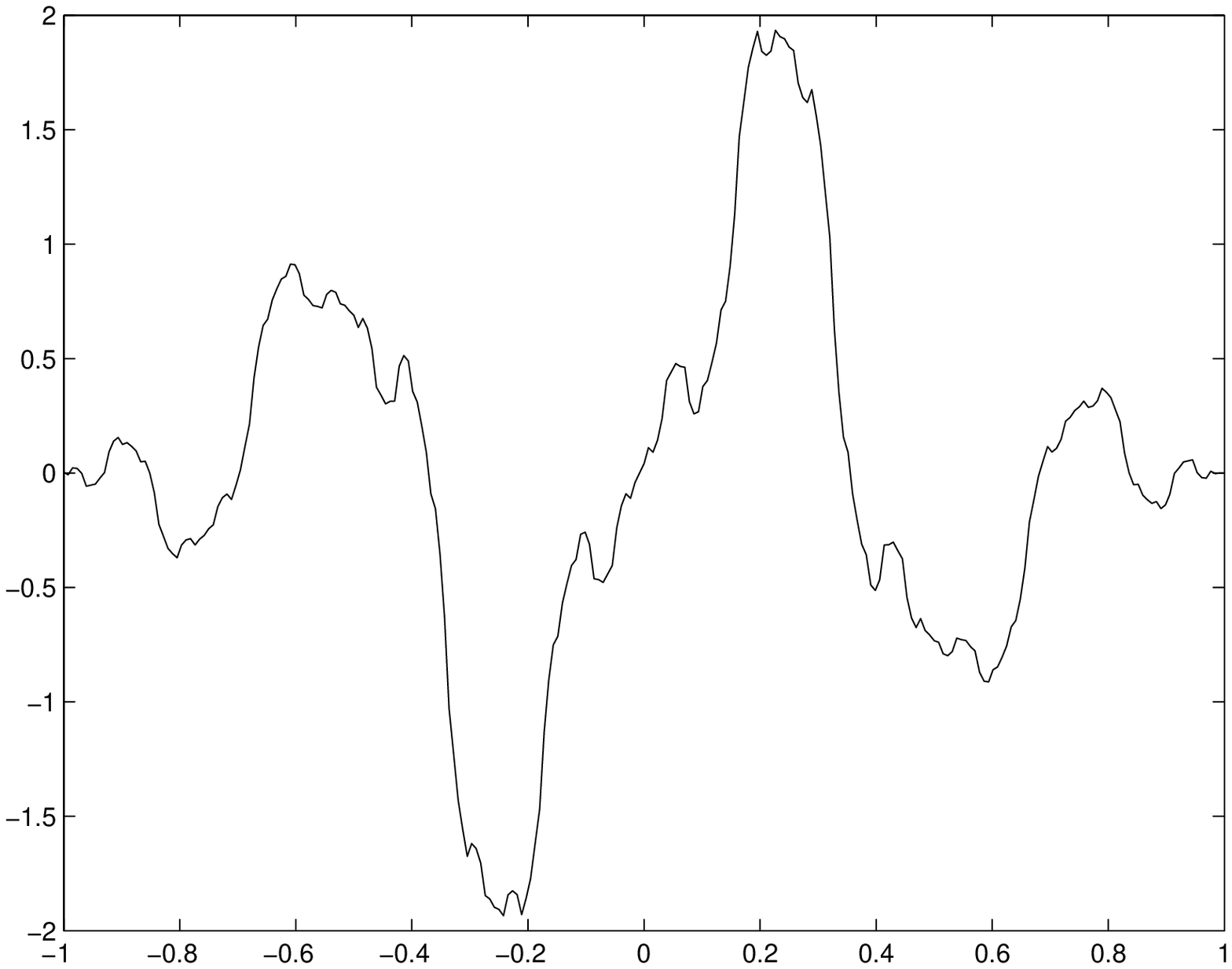,width=1.5in,height=1.5in} }}}
\begin{caption}
{The graphs of $\phi=[\phi_1,\phi_2]^T, \psi=[\psi_1,\psi_2]^T$
(top, left to right), and $\wt\phi=[\wt\phi_1,\wt\phi_2]^T,
\wt\psi=[\wt\psi_1,\wt\psi_2]^T$ (bottom, left to right) in
Example~\ref{bio:ex:2}.}
\end{caption}
\end{figure}

}\end{example}
\begin{example} \label{bio:ex:3}  {\rm

Let $\df=3, r=2$, and $a_0, \wt a_0$ be a pair of dual $\df$-filters
with symbols $\pa_0(z), \wt\pa_0(z)$ (cf.
\cite{Han.Kwon.Zhuang:2009}) given by
\[
\pa_0(z)=\frac{1}{243}\left[ \begin {array}{cc}
a_{11}(z)&a_{12}(z)\\
\noalign{\medskip} a_{21}(z)&a_{22}(z)\end {array} \right] ,\quad
 \wt\pa_0(z)
 =\frac{1}{34884}
\left[ \begin {array}{cc}
 \wt a_{11}(z)&\wt a_{12}(z)\\
\noalign{\medskip}\wt a_{21}(z) & \wt a_{22}(z)\end {array} \right].
\]
where
\[
\begin{aligned}
a_{11}(z)&=-21\,{z}^{-2}+30\,{z}^{-1}+81+14\,z-5\,{z}^{2}\\
a_{12}(z)&=60\,{z}^{-1}+84-4\,{z}^{2}+4\,{z}^{3}\\
a_{21}(z)&=4\,{z}^{-2}-4\,{z}^{-1}+84\,z+60\,{z}^{2}\\
a_{22}(z)&=-5\,{z}^{-1}+14+81\,z+30\,{z}^{2}-21 \,{z}^{3},
\end{aligned}
\] and
\[
\begin{aligned}
\wt a_{11}(z)&=
1292\,{z}^{-2}+2844\,{z}^{-1}+17496+2590\,z-1284\,{z}^{2}+1866\,{z}^{3}\\
\wt a_{12}(z)&=-4773\,{z}^{-2}+9682\,{z}^{-1}+8715-2961
\,z+386\,{z}^{2}-969\,{z}^{3}\\
\wt a_{21}(z)&=-969\,{z}^{-2}+386\,
{z}^{-1}-2961+8715\,z+9682\,{z}^{2}-4773\,{z}^{3},\\
\wt a_{22}(z)&=1866\,{z}^{-2}-1284
\,{z}^{-1}+2590+17496\,z+2844\,{z}^{2}+1292\,{z}^{3}.
\end{aligned}
\]
The low-pass filters $a_0$ and $\wt a_0$ do not satisfy
\eqref{mask:sym}. However, we can employ  a very simple  orthogonal
transform $E:= {\tiny \left[
      \begin{array}{cc}
        1 & 1 \\
        1 & -1 \\
      \end{array}
\right]}$ to $\pa_0,\wt\pa_0$ so that the symmetry in
\eqref{mask:sym} holds. That is, for $\pb_0(z):=E\pa_0(z) E^{-1}$
and $\wt\pb_0(z):=E^{-1}\wt \pa_0(z) E$, it is easy to verify that
$\pb_0$ and $\wt\pb_0$ satisfy \eqref{mask:sym} with $c_1=c_2=1/2$
and $\varepsilon_1=1,\varepsilon_2=-1$. Let $\df=d_1d_2$ with
$d_1=1$ and $d_2=3$. Construct $\pP_{
\pb_0}:=[\pb_{0;0},\pb_{0;1},\pb_{0;2}]$ and $\wt \pP_{\wt
\pb_0}:=[\wt\pb_{0;0},\wt\pb_{0;1},\wt\pb_{0;2}]$ from $\pb_0$ and
$\wt \pb_0$. Let $\pU$ be   given by
\[
\begin{small}
\pU=\left[
     \begin{array}{cccccc}
       z^{-1} & 0 & z^{-1} & 0 & 0 & 0 \\
       0 & z^{-1} & 0 & z^{-1} & 0 & 0 \\
       1 & 0 & -1 & 0 & 0 & 0 \\
       0 & 1 & 0 & -1 & 0 & 0 \\
       0 & 0 & 0 & 0 & 1 & 0 \\
       0 & 0 & 0 & 0 & 0 & 1 \\
     \end{array}
   \right]
\end{small}
\]
and define $\wt\pU:=(\pU^*)^{-1}$. Let $\pP:=\pP_{\pb_0}\pU$ and
$\wt\pP:=\wt\pP_{\wt \pb_0}\wt\pU$. Then we have $
\sym\pP=\sym\wt\pP=[z^{-1},-z^{-1}]^T[1,-1,-1,1,1,-1]$ and
$\pP,\wt\pP$ are given by
\[
\begin{aligned}
\pP&=c\left[
 \begin{array}{cccccc}
t_{11}(1+\frac1z)&t_{12}(1-\frac1z)&t_{13}(1-\frac1z)&t_{14}&t_{15}(1+\frac1z)&t_{16}(1-\frac1z)
\\ \noalign{\vspace{0.05in}}
t_{21}(1-\frac1z)&t_{22}(1+\frac1z)&t_{23}(1+\frac1z)&t_{24}(1-\frac1z)&t_{25}(1-\frac1z)&t_{26}(1+\frac1z)
\end{array}
\right],\\
\wt\pP&=\wt c\left[
 \begin{array}{cccccc}
\wt t_{11}(1+\frac1z)&\wt t_{12}(1-\frac1z)&\wt
t_{13}(1-\frac1z)&\wt t_{14}&\wt t_{15}(1+\frac1z)&\wt
t_{16}(1-\frac1z)
\\ \noalign{\vspace{0.05in}}
\wt t_{21}(1-\frac1z)&\wt t_{22}(1+\frac1z)&\wt
t_{23}(1+\frac1z)&\wt t_{24}(1-\frac1z)&\wt t_{25}(1-\frac1z)&\wt
t_{26}(1+\frac1z)\\
\end{array}
\right],
\end{aligned}
\]
where $c=\frac{1}{486},  \wt  c = \frac{3}{34884}$ and $t_{jk}$'s,
$\wt t_{jk}$'s are constants defined as follows:
\[
\begin{small}
\begin{aligned}
t_{11}&=162;& t_{12}&=34; & t_{13}&=-196;& t_{14}&=0;&
t_{15}&=81;& t_{16}&=29;&\\
t_{21}&=-126;& t_{22}&=-14; & t_{13}&=176;& t_{24}&=-36;&
t_{15}&=-99;& t_{16}&=-31;&\\
\wt t_{11}&=5814;& \wt t_{12}&=-1615; & \wt t_{13}&=-7160;& \wt
t_{14}&=0;&
\wt t_{15}&=5814;& \wt t_{16}&=2584;&\\
\wt t_{21}&=-5551;& \wt t_{22}&=5808; &\wt  t_{13}&=7740;& \wt
 t_{24}&=-1358;&
\wt t_{15}&=-6712;& \wt t_{16}&=-4254.&\\
\end{aligned}
\end{small}
\]
Applying Algorithm~2,  we obtain  $\pP_e$ and $\wt\pP_e$ as follows:
\[
\pP_e=c\left[
 \begin{array}{cccccc}
t_{11}(1+\frac1z)&t_{12}(1-\frac1z)&t_{13}(1-\frac1z)&t_{14}&t_{15}(1+\frac1z)&t_{16}(1-\frac1z)
\\ \noalign{\vspace{0.05in}}
t_{21}(1-\frac1z)&t_{22}(1+\frac1z)&t_{23}(1+\frac1z)&t_{24}(1-\frac1z)&t_{25}(1-\frac1z)&t_{26}(1+\frac1z)\\
\hline
%
%
t_{31}(1+\tfrac1z)&t_{32}(1-\tfrac1z)&t_{33}(1-\tfrac1z)&t_{34}(1+\tfrac1z)&t_{35}(1+\tfrac1z)&t_{36}(1-\tfrac1z)
\\ \noalign{\vspace{0.05in}}
t_{41}&0&0&t_{44}&t_{45}&0\\
\hline
%
%
0 & t_{52} & t_{53} & 0 & 0 & t_{56}\\
t_{61}(1-\frac1z)& t_{62}(1+\frac1z)& t_{63}(1+\frac1z) & t_{64}(1-\frac1z) & t_{65}(1-\frac1z)& t_{66}(1+\frac1z)\\
\end{array}
\right],
\]
where all $t_{jk}$'s are constants given by:
\[
\begin{aligned}
t_{31}&=24;& t_{32}&=\frac{472}{27}; & t_{33}&=-\frac{148}{27};\\
t_{34}&=-36; & t_{35}&=-24; & t_{36}& =
-\frac{112}{27};\\
t_{41}&=\frac{109998}{533}; & t_{44}&=\frac{94041}{533};&
t_{45}&=-\frac{109989}{533};\\
t_{52}&=406c_0; & t_{53}&=323c_0;& t_{56}&=1142c_0;
&c_0&=\frac{1609537}{13122};\\
t_{61}&=24210c_1;& t_{62}&=14318c_1; & t_{63}&=-11807c_1;&
t_{64}&=-26721c_1;\\
t_{65}&=-14616c_1; & t_{66}& =
-1934c_1; &c_1&=200/26163.\\
\end{aligned}
\]
And
\[
\wt\pP_e=\wt c\left[
 \begin{array}{cccccc}
\wt t_{11}(1+\frac1z)&\wt t_{12}(1-\frac1z)&\wt
t_{13}(1-\frac1z)&\wt t_{14}&\wt t_{15}(1+\frac1z)&\wt
t_{16}(1-\frac1z)
\\ \noalign{\vspace{0.05in}}
\wt t_{21}(1-\frac1z)&\wt t_{22}(1+\frac1z)&\wt
t_{23}(1+\frac1z)&\wt t_{24}(1-\frac1z)&\wt t_{25}(1-\frac1z)&\wt
t_{26}(1+\frac1z)\\
\hline
%
%
\wt t_{31}(1+\tfrac1z)&\wt t_{32}(1-\tfrac1z)&\wt
t_{33}(1-\tfrac1z)&\wt t_{34}(1+\tfrac1z)&\wt t_{35}(1+\tfrac1z)&\wt
t_{36}(1-\tfrac1z)
\\ \noalign{\vspace{0.05in}}
\wt t_{41}&0&0&\wt t_{44}&\wt  t_{45}&0\\
\hline
%
%
0 & \wt t_{52} & \wt t_{53} & 0 & 0 &\wt  t_{56}\\
\wt t_{61}(1-\frac1z)& \wt t_{62}(1+\frac1z)& \wt t_{63}(1+\frac1z) & \wt t_{64}(1-\frac1z) &\wt  t_{65}(1-\frac1z)&\wt  t_{66}(1+\frac1z)\\
\end{array}
\right],
\]
where all $\wt t_{jk}$'s are constants given by:
\[
\begin{aligned}
\wt t_{31}&=3483\wt c_0;&\wt  t_{32}&=37427\wt c_0; & \wt
t_{33}&=4342\wt c_0;& \wt t_{34}&=-12222\wt c_0;\\  \wt
t_{35}&=-3483\wt c_0; & \wt t_{36}& =-7267; &
\wt c_0&=\frac{8721}{4264};\\
\wt t_{41}&=5814; & \wt t_{44}&=11628;&
\wt t_{45}&=-11628;\\
\wt t_{52}&=3\wt c_1; &\wt t_{53}&=2\wt c_1;& \wt t_{56}&=10\wt c_1;
&\wt c_1&=\frac{12680011}{243};\\
\end{aligned}
\]
\[
\begin{aligned}
\wt t_{61}&=18203\wt c_2;& \wt t_{62}&=101595\wt c_2; &\wt
t_{63}&=1638\wt c_2;&
\wt t_{64}&=-33950\wt c_2;\\
\wt t_{65}&=-10822\wt c_2; & \wt t_{66}& =
-36582\wt c_2; &\wt c_2&=\frac{26163}{213200}.\\
\end{aligned}
\]
Note that  $\pP_e$ and $\wt\pP_e$ satisfy
\[
\sym\pP_e=\sym\pP_e=[z^{-1},-z^{-1},z^{-1},1,-1,-z^{-1}]^T[1,-1,-1,1,1,-1].
\]From the polyphase matrices $\PP:=\pP_e \wt\pU^*$ and
$\wt\PP:=\wt\pP_e\pU^*$, we derive high-pass filters $\pb_1,\pb_2$
and $\wt\pb_1, \wt \pb_2$ as follows:
\[
 \pb_1(z)=\left[
           \begin{array}{cc}
             b^1_{11}(z) & b^1_{12}(z) \\
              b^1_{21}(z) &  b^1_{22}(z)\\
           \end{array}
         \right],
 \pb_2(z)=\left[
           \begin{array}{cc}
             b^2_{11}(z) &  b^2_{12}(z) \\
              b^2_{21}(z) &  b^2_{22}(z)\\
           \end{array}
         \right],
\]
where
\[
\begin{aligned}
b^1_{11}(z)&={\frac {199}{6561}}+{\frac {125}{6561}}{z}^{3}-{\frac
{4}{81}}{z}^ {2}+{\frac {199}{6561}}z-{\frac {4}{81}}{z}^{-1}+{\frac
{125}{6561 }}{z}^{-2}
;\\
b^1_{12}(z)&=-{\frac {361}{6561}}-{\frac {125}{6561}}{z}^{3}-{\frac
{56}{6561}} {z}^{2}+{\frac {361}{6561}}z+{\frac
{56}{6561}}{z}^{-1}+{\frac { 125}{6561}}{z}^{-2}
;\\
b^1_{21}(z)&={\frac {679}{3198}}{z}^{3}+{\frac {679}{3198}} z-{\frac
{679}{1599}}{z}^{2};
\\
b^1_{22}(z)&={\frac {387}{2132}}{z}^{3}-{\frac {387}{2132}}z
;\\
%
%
b^2_{11}(z)&=c_3(323{z}^{3}-323z);\\
b^2_{12}(z)&=c_3(406{z}^{3}+2284{z}^{2}+406z);\\
b^2_{21}(z)&=c_4(
-36017+12403\,{z}^{3}-29232\,{z}^{2}+36017\,z+29232\,{z}^{-1}-12403\,{z}^{-2});
\\
b^2_{22}(z)&=c_4(41039-12403\,{z}^{3}-3868\,{z}^{2}
+41039\,z-3868\,{z}^{-1}-12403\,{z}^{-2});
\\
 c_3&=\frac{27}{3219074};\,\, c_4=\frac{50}{6357609}.
\end{aligned}
\]
And
\[
 \wt\pb_1(z)=\left[
           \begin{array}{cc}
             \wt b^1_{11}(z) & \wt b^1_{12}(z) \\
              \wt b^1_{21}(z) &  \wt  b^1_{22}(z)\\
           \end{array}
         \right],
\wt \pb_2(z)=\left[
           \begin{array}{cc}
             \wt b^2_{11}(z) &\wt   b^2_{12}(z) \\
             \wt  b^2_{21}(z) & \wt  b^2_{22}(z)\\
           \end{array}
         \right],
\]
where
\[
\begin{aligned}
\wt b^1_{11}(z)&=-{\frac {859}{17056}}+{\frac
{7825}{17056}}\,{z}^{3}-{\frac {3483}{ 8528}}\,{z}^{2}-{\frac
{859}{17056}}\,z-{\frac {3483}{8528}}\,{z}^{-1} +{\frac
{7825}{17056}}\,{z}^{-2}
;\\
\wt b^1_{12}(z)&=-{\frac {49649}{17056}}+{\frac
{25205}{17056}}\,{z}^{3}-{\frac {559}{ 656}}\,{z}^{2}+{\frac
{49649}{17056}}\,z+{\frac {559}{656}}\,{z}^{-1}- {\frac
{25205}{17056}}\,{z}^{-2}
;\\
\wt b^1_{21}(z)&=\frac16({z}^{3}+z-2{z}^{2});
%
\quad\wt b^1_{22}(z)=\frac13({z}^{3}-z)
;\\
%
%
\wt b^2_{11}(z)&=2\wt c_3({z}^{3}-z);\\
\wt b^2_{12}(z)&=\wt c_3(3{z}^{3}+10{z}^{2}+3z);\,\,\wt c_3=\frac{39257}{26244};\\
\wt b^2_{21}(z)&= -{\frac {9939}{170560}}+{\frac
{59523}{852800}}\,{z}^{3}-{\frac {16233}{426400}}\,{z}^{2}+{\frac
{9939}{ 170560}}\,z+{\frac {16233}{426400}}\,{z}^{-1}-{\frac
{59523}{852800}} \,{z}^{-2};
\\
\wt b^2_{22}(z)&={\frac {81327}{170560}}+{\frac
{40587}{170560}}\,{z}^{3}-{ \frac {4221}{32800}}\,{z}^{2}+{\frac
{81327}{170560}}\,z-{\frac {4221} {32800}}\,{z}^{-1}+{\frac
{40587}{170560}}\,{z}^{-2}.
\end{aligned}
\]
Then the high-pass filters $\pb_1,\pb_2$ and $\wt \pb_1, \wt \pb_2$
satisfy \eqref{bio:highpass:sym} with $c^1_1=c^1_2=1/2$,
$\varepsilon^1_1=1, \varepsilon^1_2=1$ and $c^2_1=c^2_2=3/2$,
$\varepsilon^1_1=-1, \varepsilon^1_2=-1$, respectively.

Let $\pa_1,\pa_2$ and $\wt\pa_1,\wt\pa_2$ be high-pass filters
constructed from $\pb_1,\pb_2$ and $\wt \pb_1, \wt \pb_2$ by
$\pa_1(z):=E^{-1}\pb_1(z)E, \pa_2:=E^{-1}\pb_2E$ and
$\wt\pa_1(z):=E\wt\pb_1(z)E^{-1}, \wt\pa_2:=E\wt\pb_2E^{-1}$.

See Figure~3.2 for the graphs of the $3$-refinable function vectors
$\phi,\wt\phi$ associated with the low-pass filters
$\pa_0,\wt\pa_0$, respectively, and the biorthogonal multiwavelet
function vectors $\psi^1,\psi^2$ and $\wt\psi^1,\wt\psi^2$
associated with the high-pass filters $\pa_1,\pa_2$ and
$\wt\pa_1,\wt\pa_2$, respectively. Also, see Figure~3.3 for the
graphs of the $3$-refinable function vectors $\eta, \wt\eta$
associated with the low-pass filters $\pb_0, \wt\pb_0$,
respectively, and the biorthogonal multiwavelet function vectors
$\zeta^1,\zeta^2$ and $\wt\zeta^1, \wt\zeta^2$ associated with the
high-pass filters $\pb_1,\pb_2$ and $\wt\pb_1,\wt\pb_2$,
respectively.

\begin{figure}[th] \label{bio:fig:ex21}
\centerline{\scalebox{0.75}{
\hbox{\epsfig{file=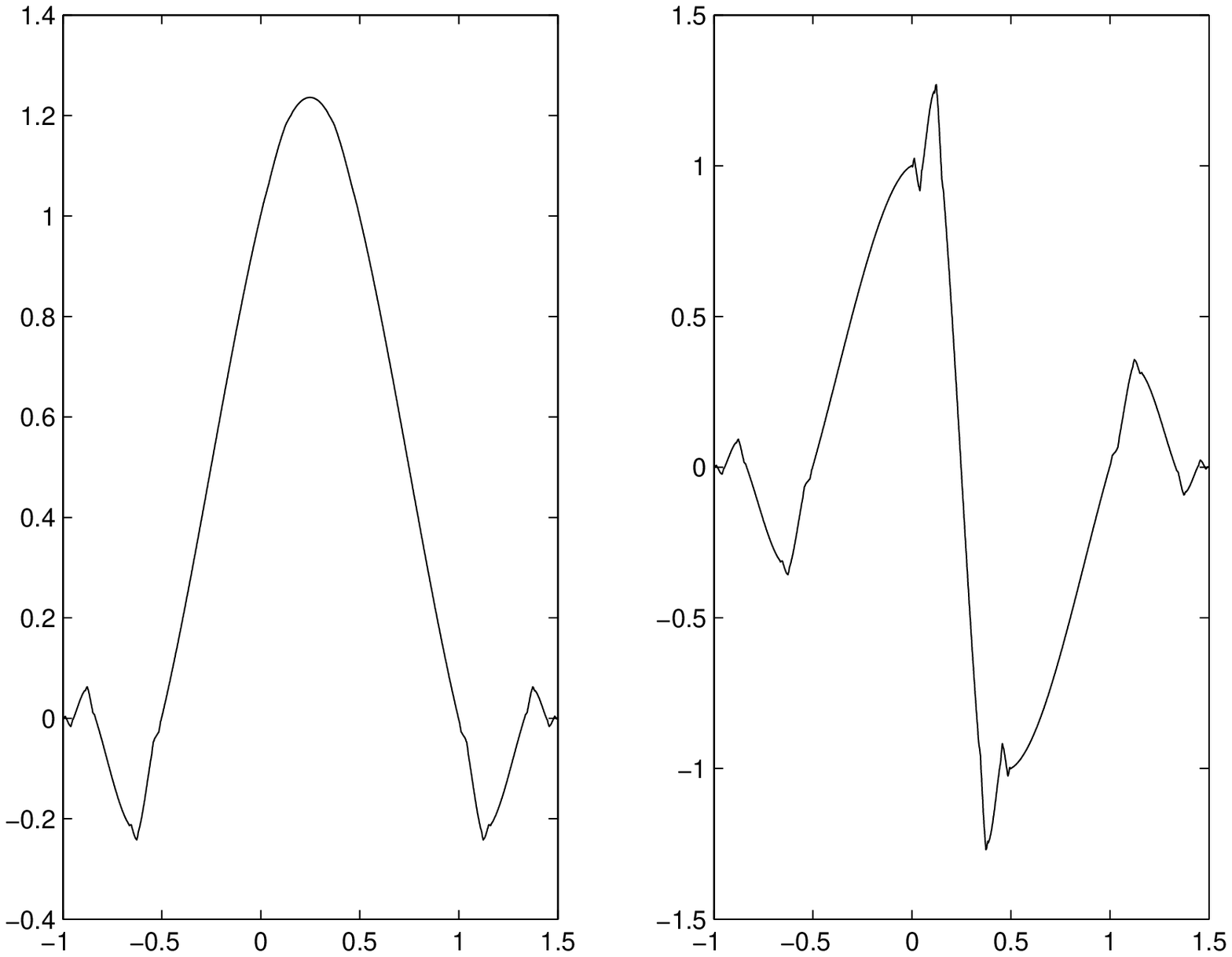,width=2.1in,height=1.5in}
\epsfig{file=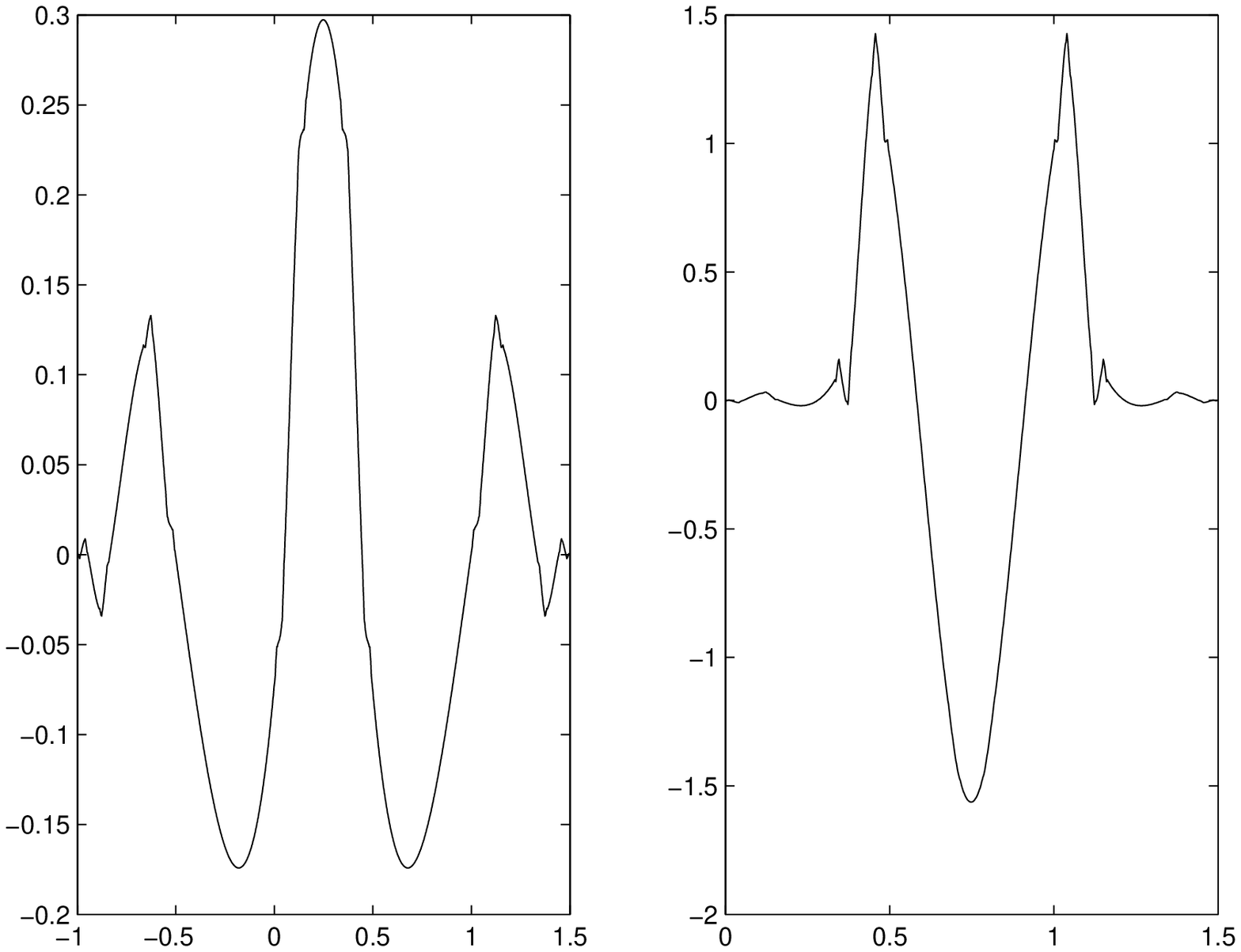,width=2.1in,height=1.5in}
\epsfig{file=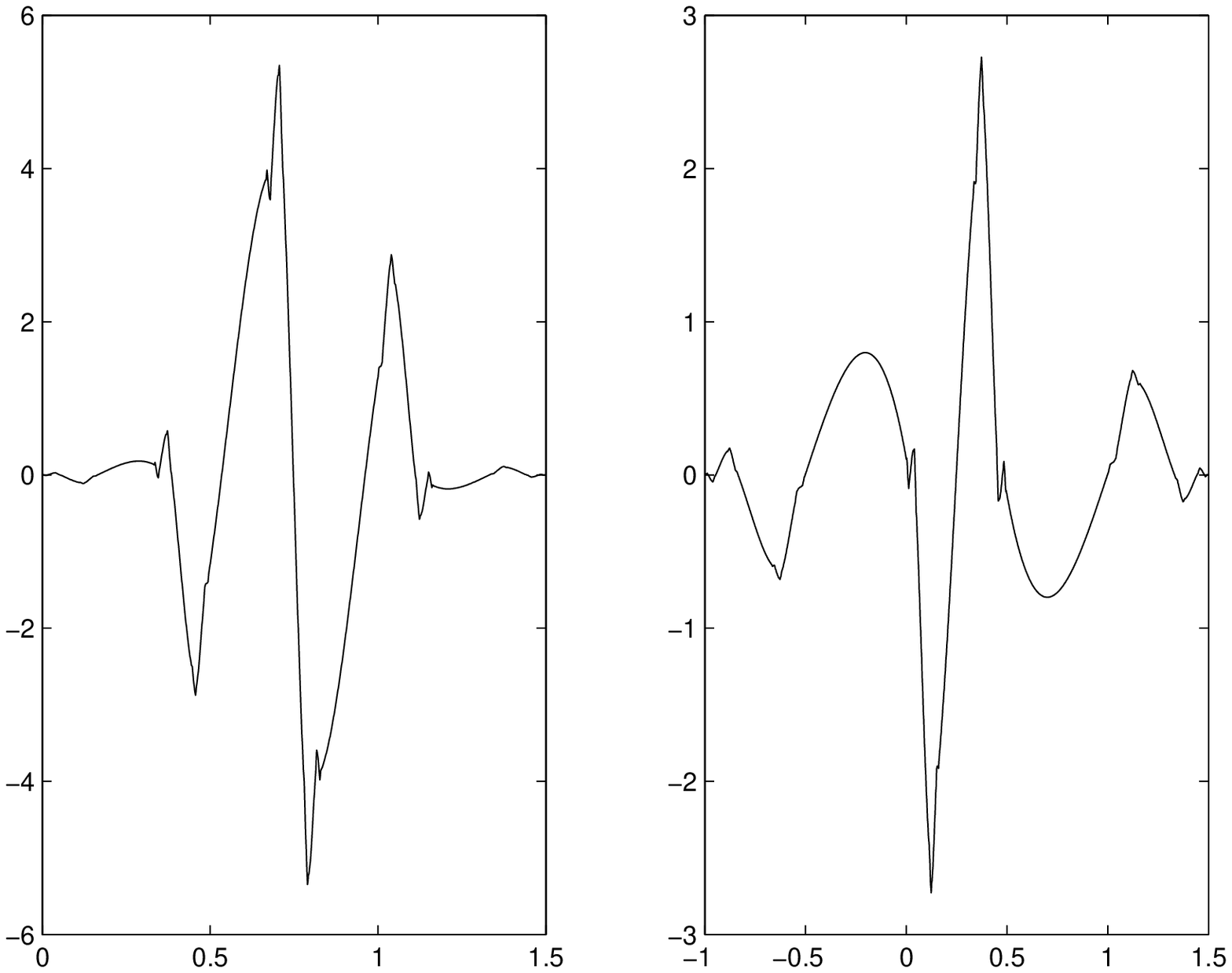,width=2.0in,height=1.5in} }}}
\centerline{\scalebox{0.75}{
\hbox{\epsfig{file=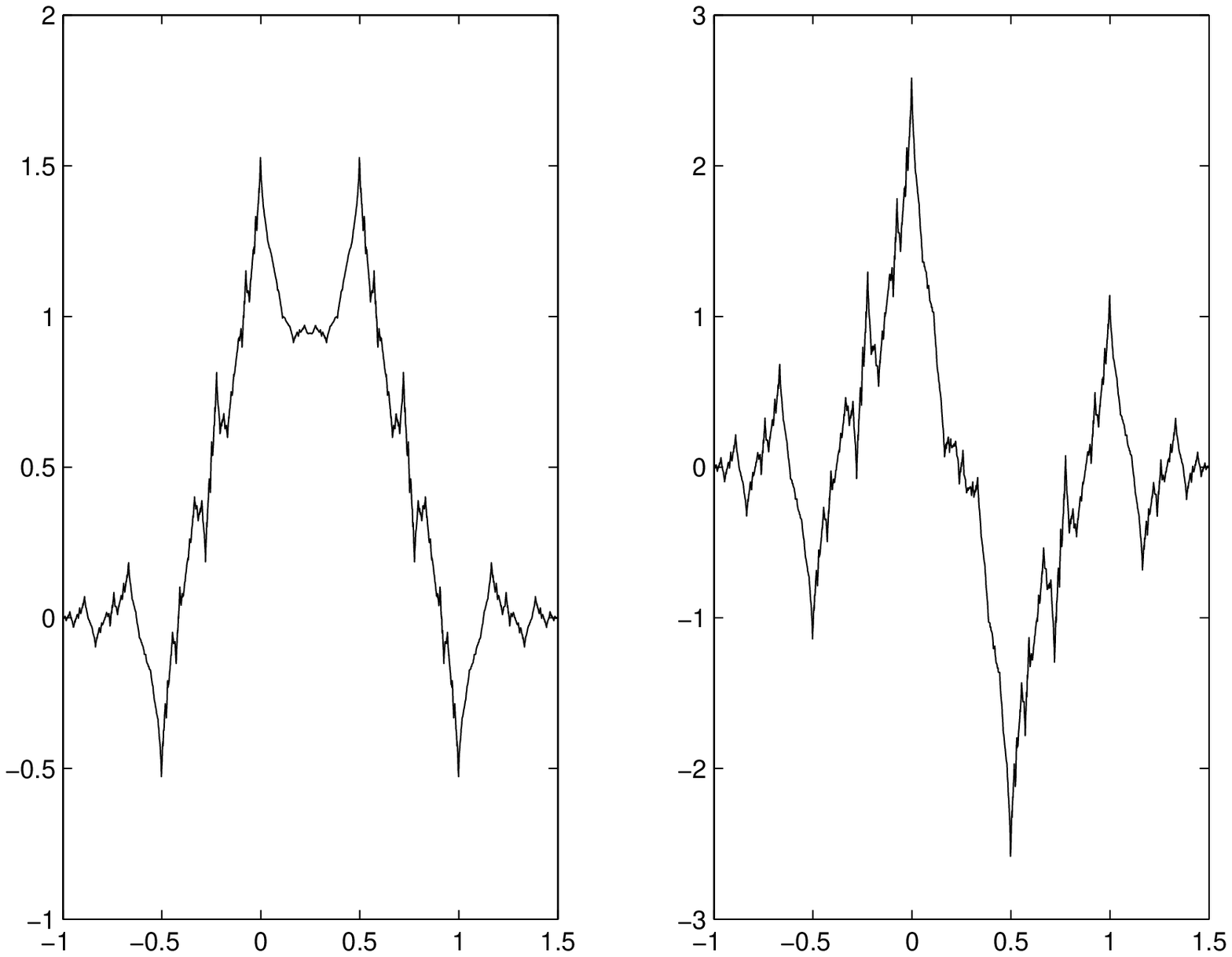,width=2.1in,height=1.5in}
\epsfig{file=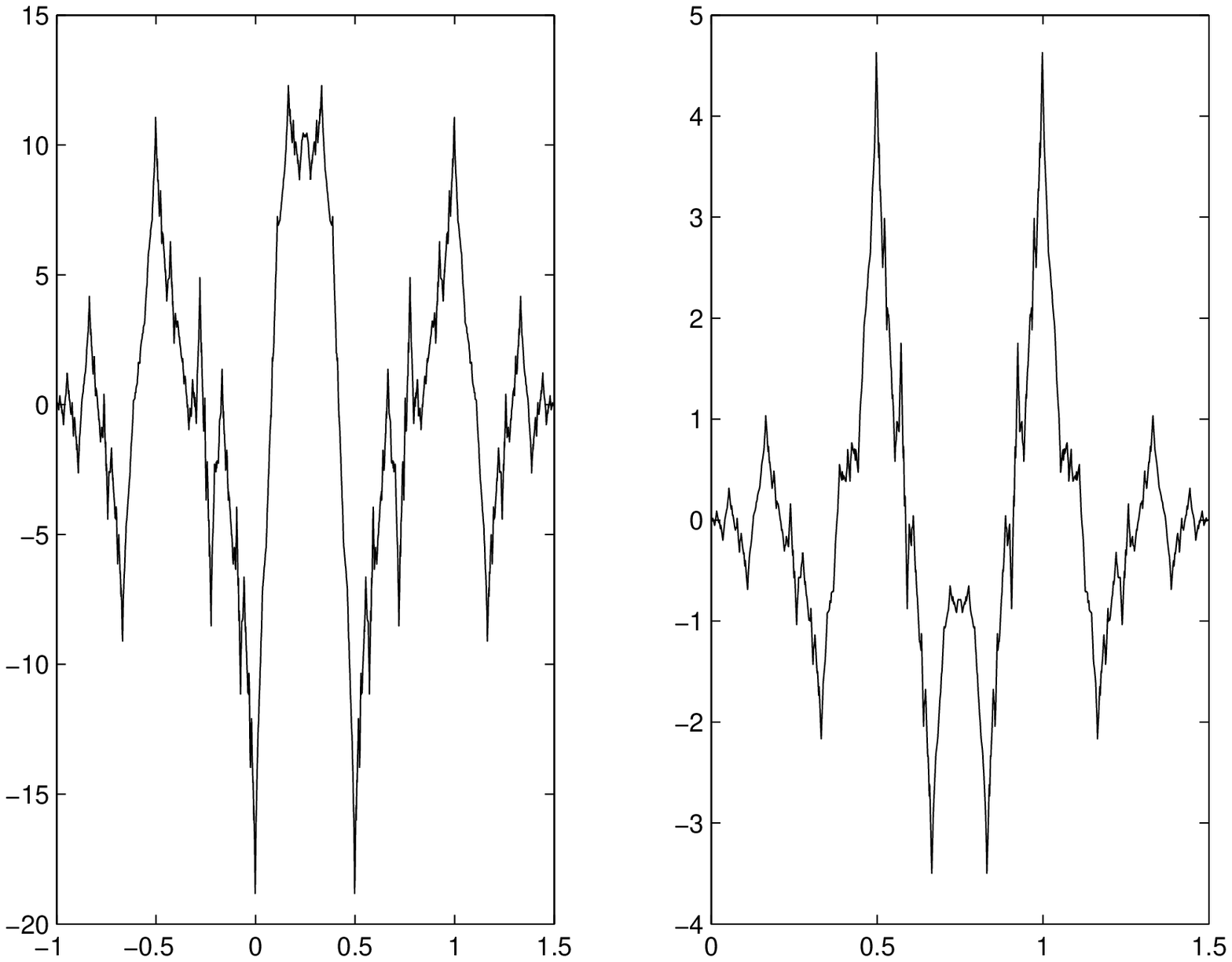,width=2.1in,height=1.5in}
\epsfig{file=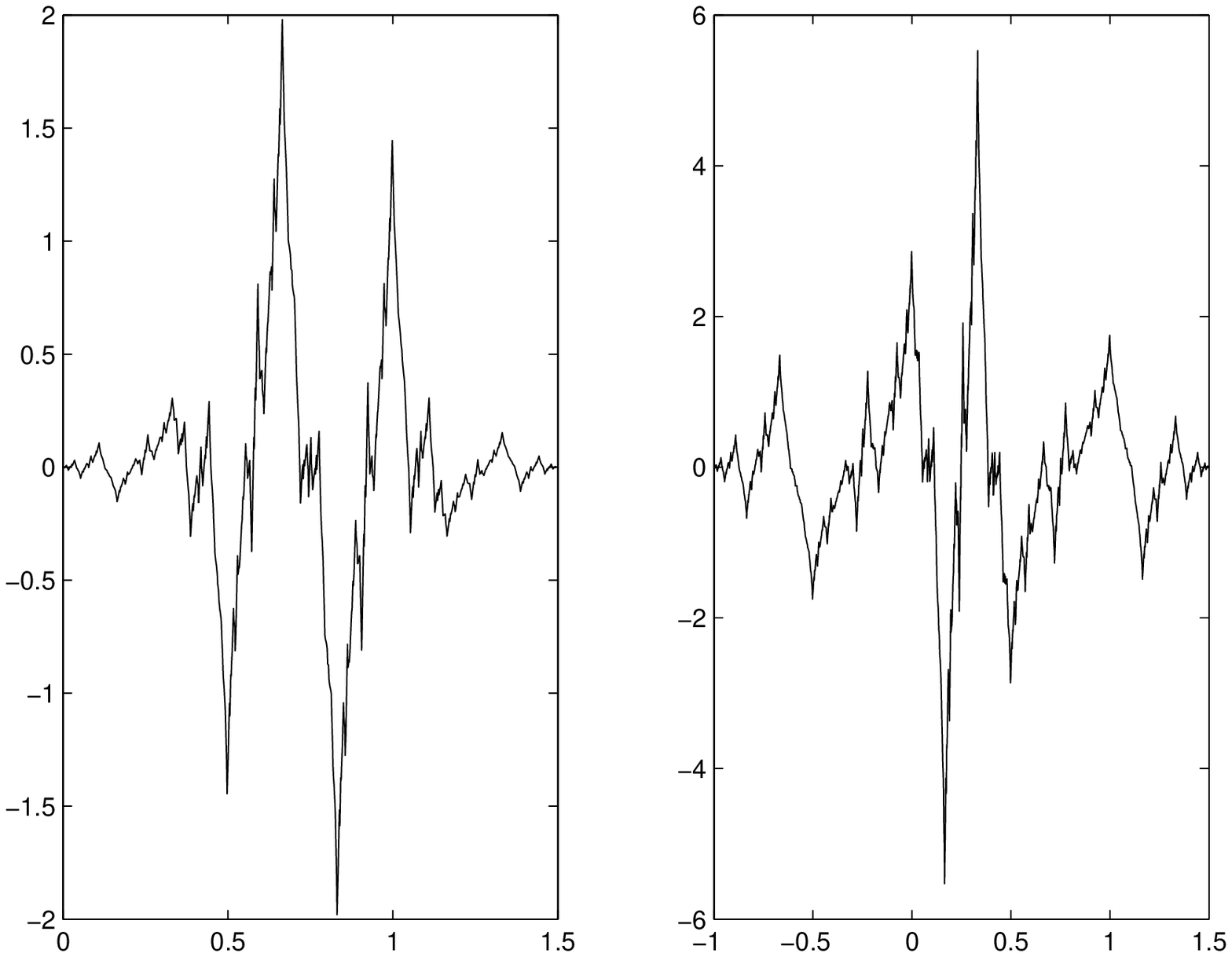,width=2.1in,height=1.5in} }}}
\begin{caption}
{The graphs of $\eta=[\eta_1,\eta_2]^T$,
$\zeta^1=[\zeta^1_1,\zeta^1_2]^T$, and
$\zeta^2=[\zeta^2_1,\zeta^2_2]^T$ (top, left to right), and
$\wt\eta=[\wt\eta_1,\wt\eta_2]^T$,
$\wt\zeta^1=[\wt\zeta^1_1,\wt\zeta^1_2]^T$, and
$\wt\zeta^2=[\wt\zeta^2_1,\wt\zeta^2_2]^T$ (bottom, left to right).}
\end{caption}
\end{figure}

\begin{figure}[th] \label{bio:fig:ex22}
\centerline{\scalebox{0.75}{
\hbox{\epsfig{file=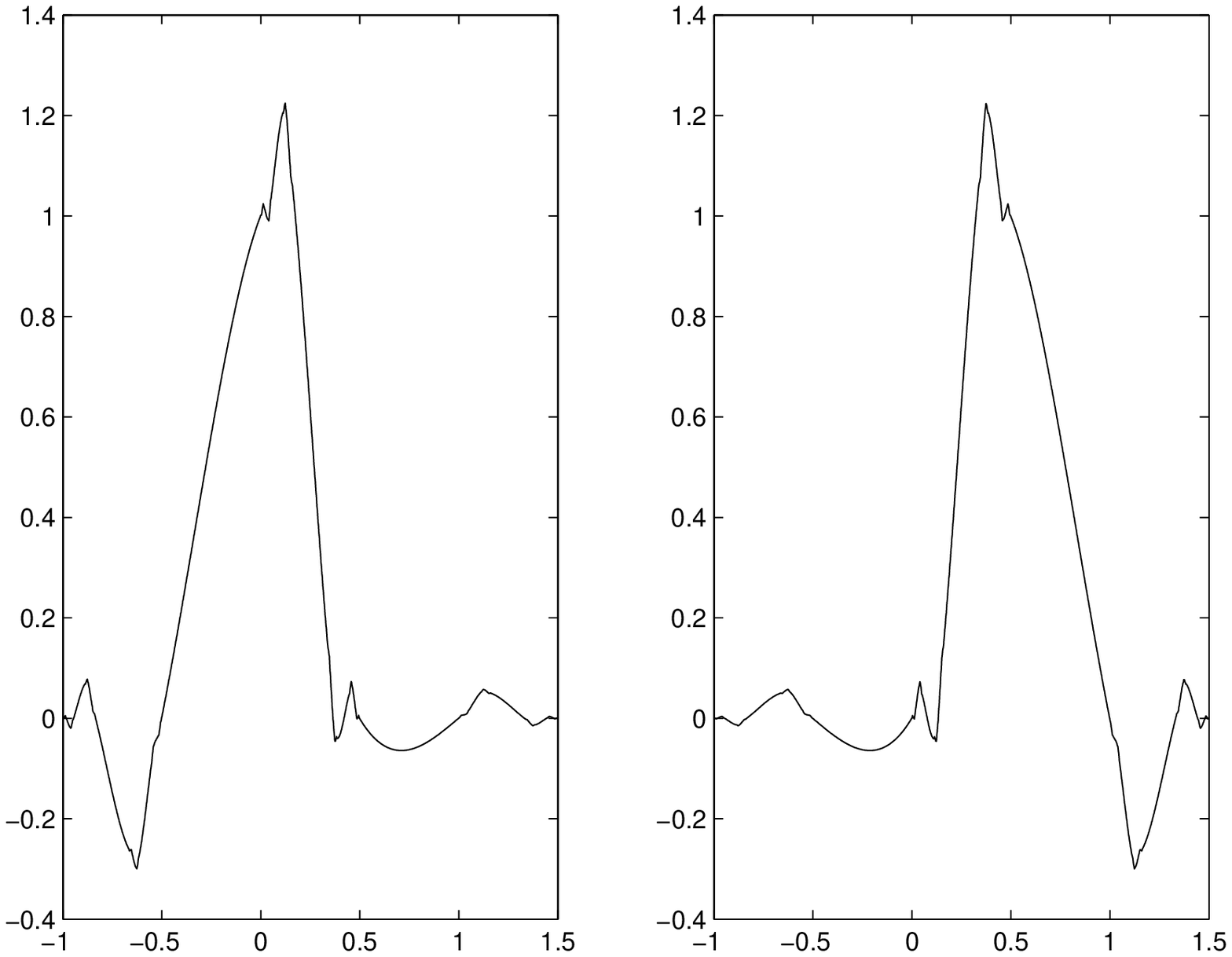,width=2.1in,height=1.5in}
\epsfig{file=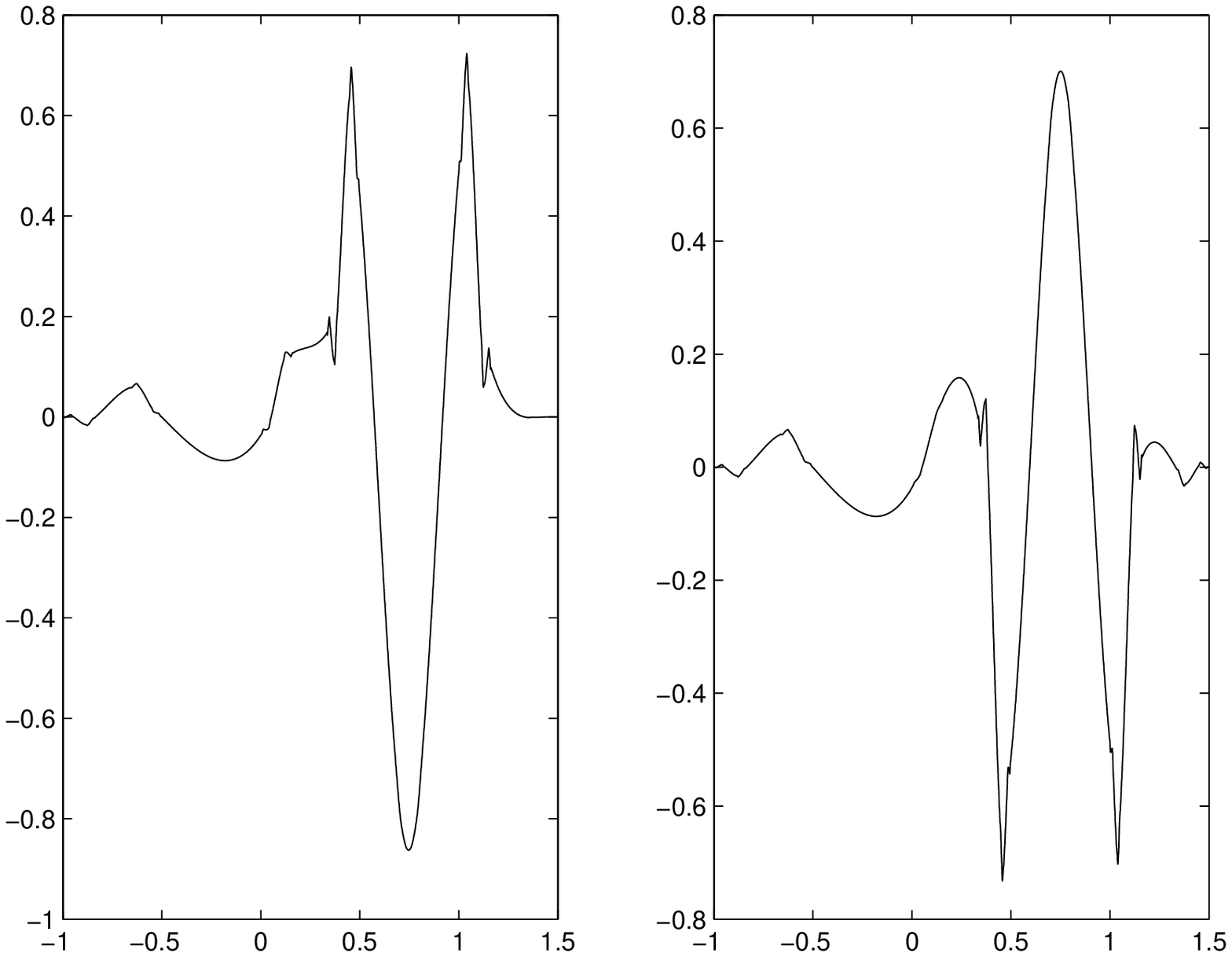,width=2.1in,height=1.5in}
\epsfig{file=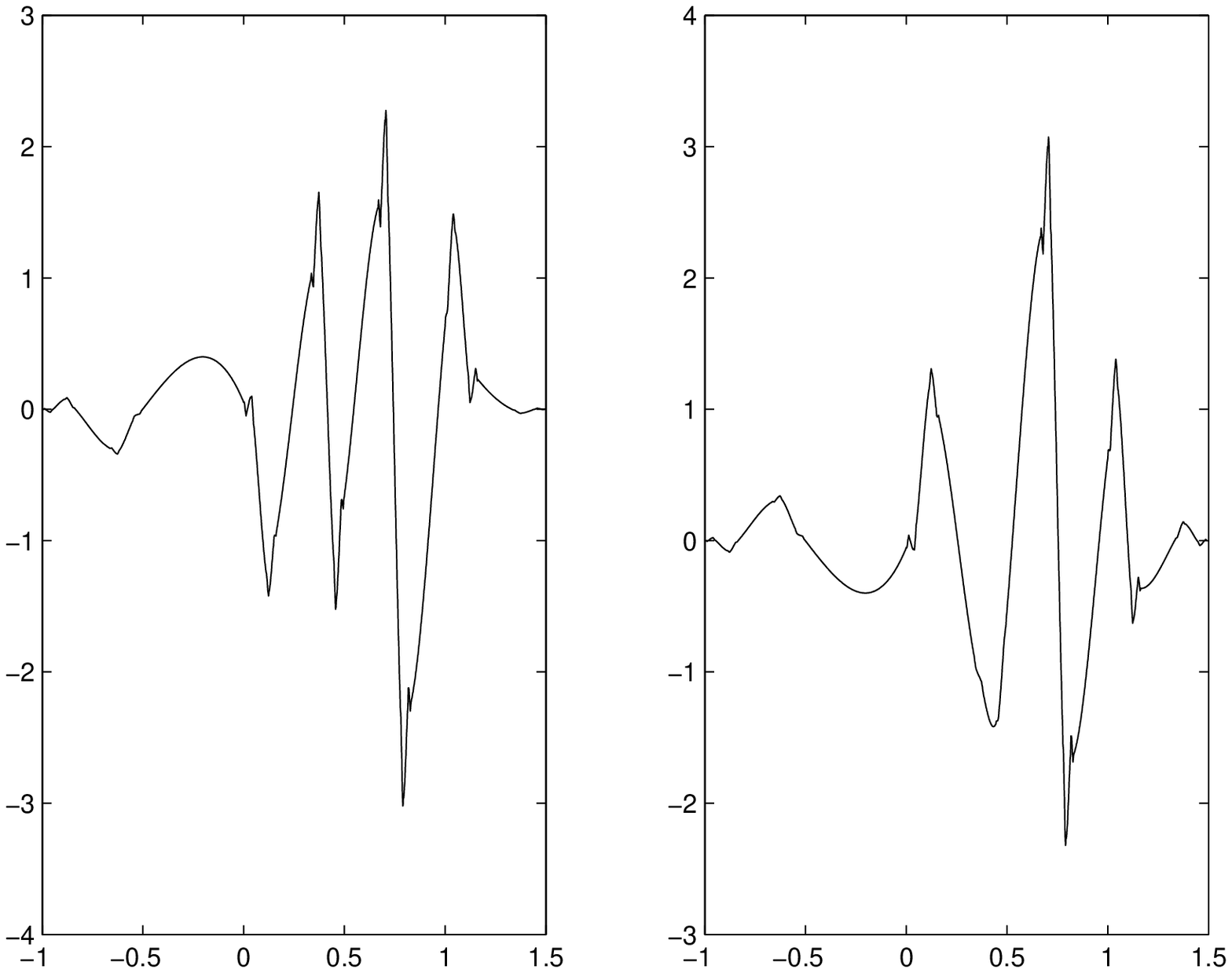,width=2.1in,height=1.5in} }}}
\centerline{\scalebox{0.75}{
\hbox{\epsfig{file=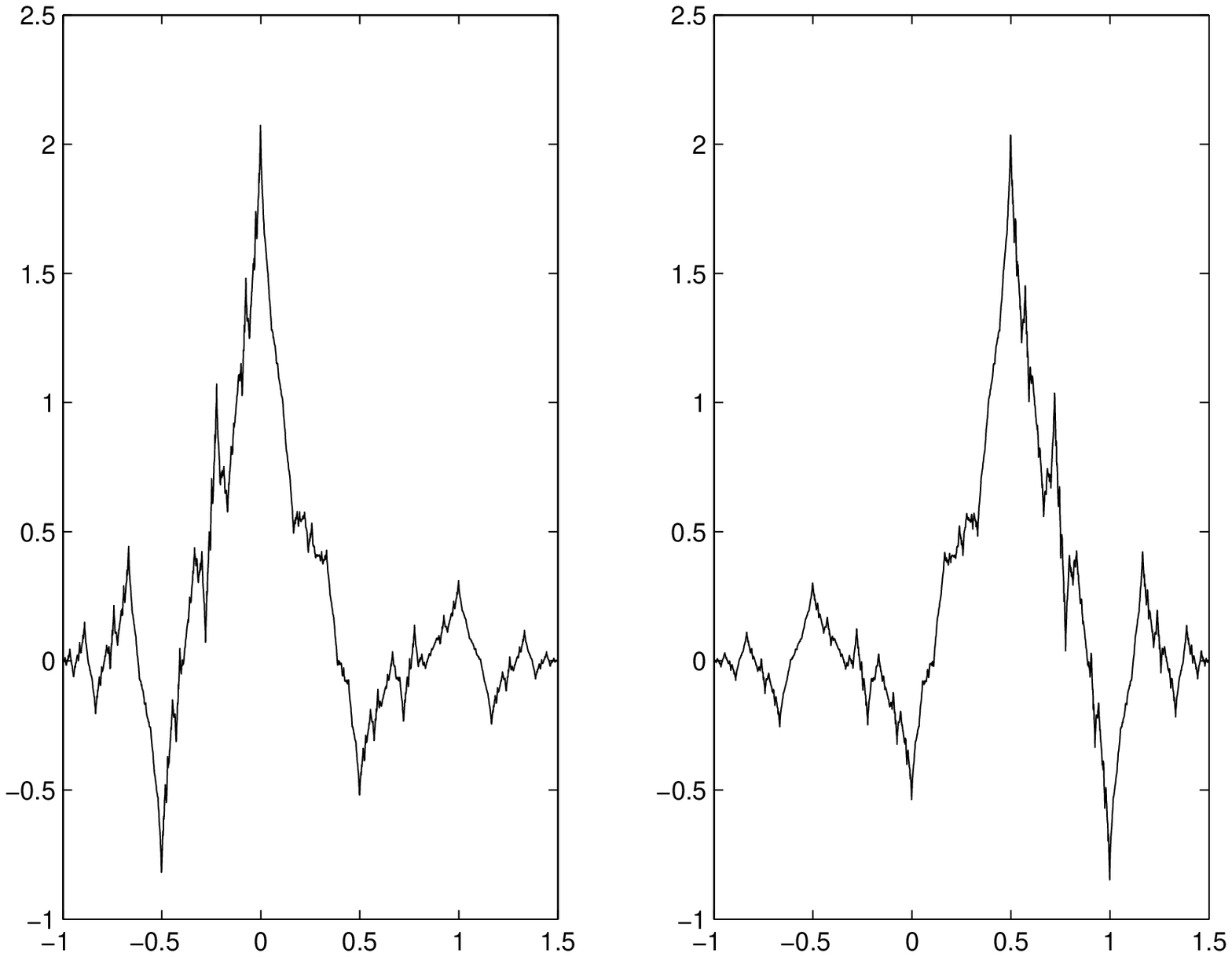,width=2.1in,height=1.5in}
\epsfig{file=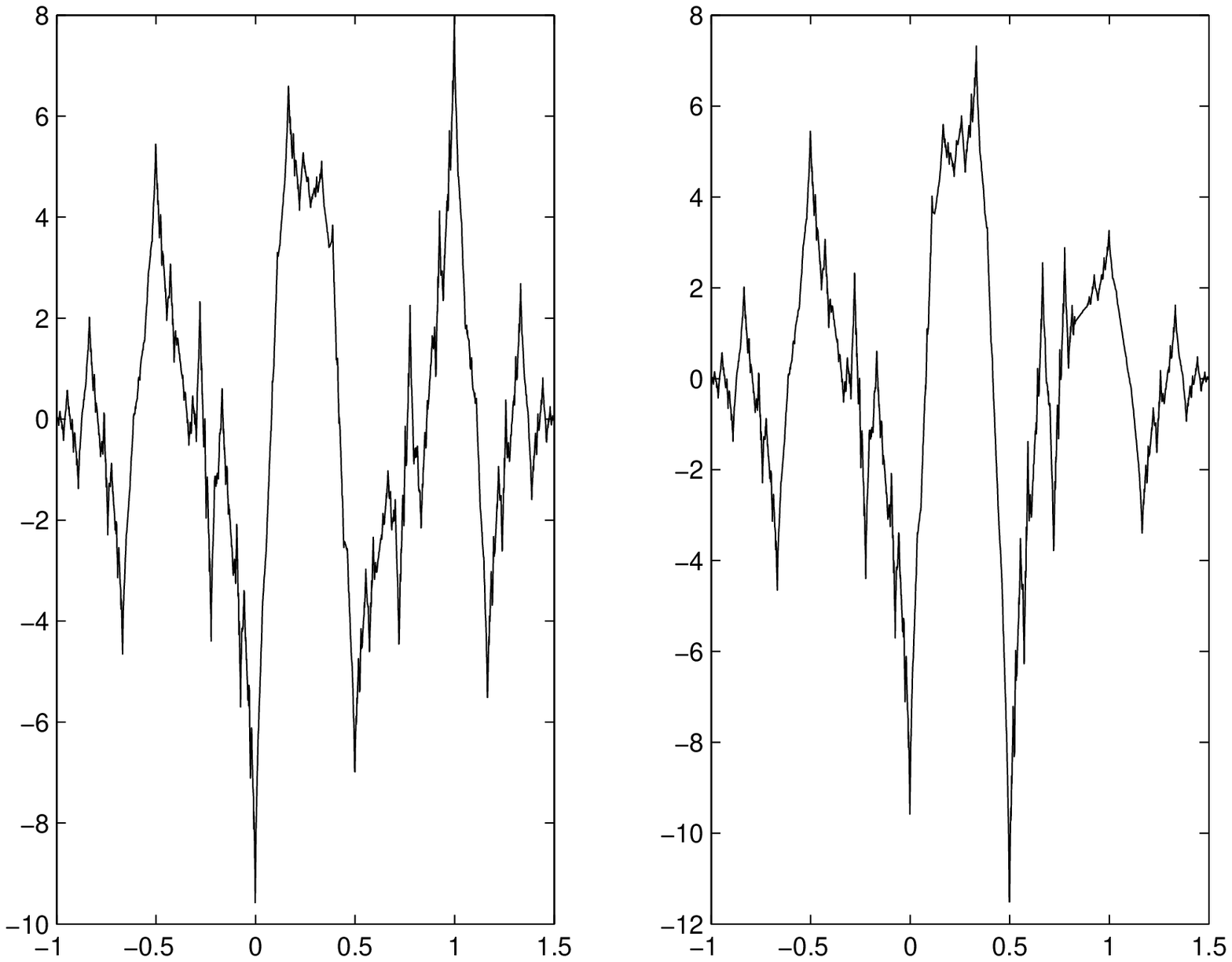,width=2.1in,height=1.5in}
\epsfig{file=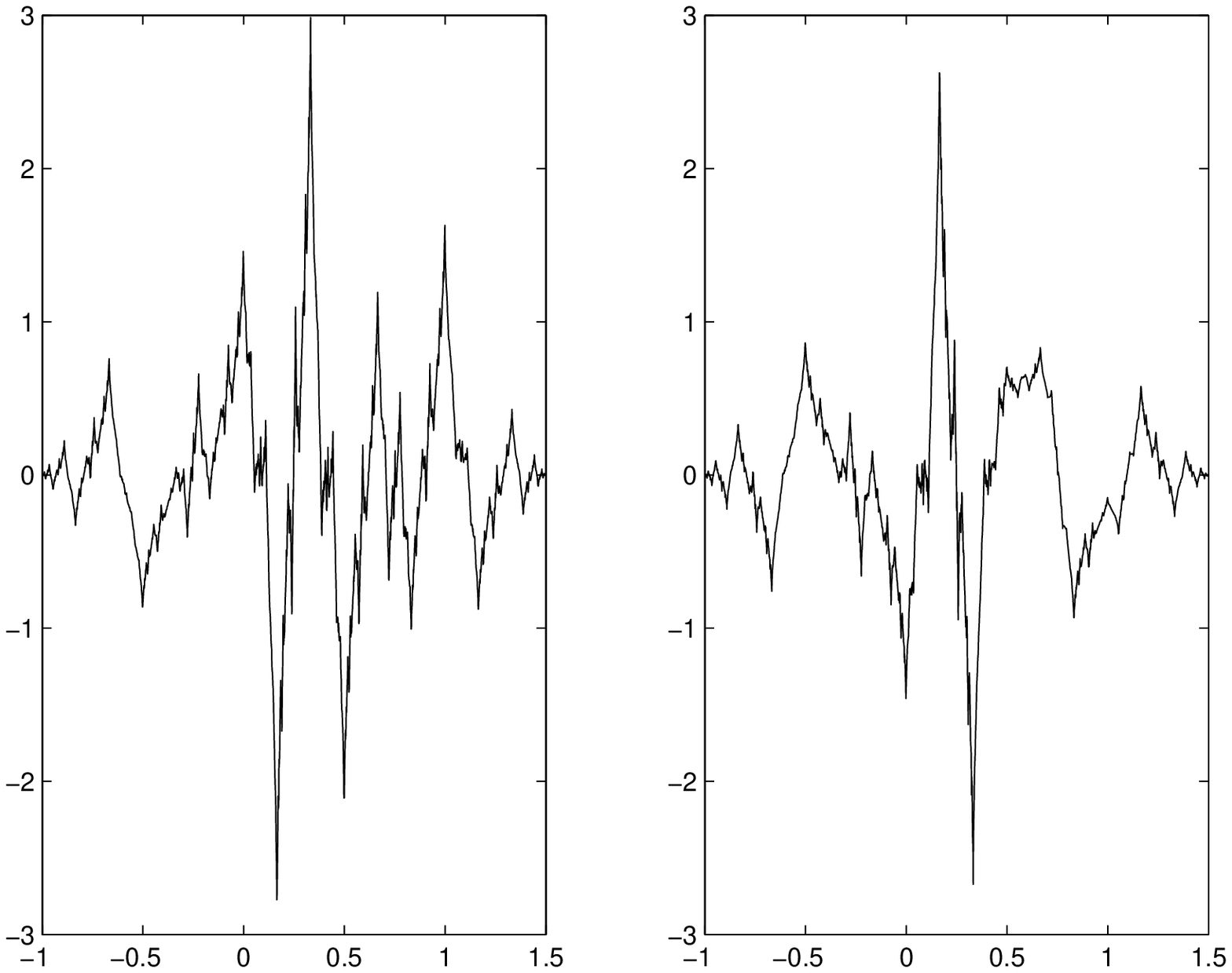,width=2.1in,height=1.5in} }}}
\begin{caption}
{The graphs of $\phi=[\phi_1,\phi_2]^T$,
$\psi^1=[\psi^1_1,\psi^1_2]^T$, and $\psi^2=[\psi^2_1,\psi^2_2]^T$
(top, left to right), and $\wt\phi=[\wt\phi_1,\wt\phi_2]^T$,
$\wt\psi^1=[\wt\psi^1_1,\wt\psi^1_2]^T$, and
$\wt\psi^2=[\wt\psi^2_1,\wt\psi^2_2]^T$ (bottom, left to right).}
\end{caption}
\end{figure}

}\end{example}

%
%
%
%
%
%
%
%
%
%
%

\section{Conclusions and Remarks}
In this paper, we study the  matrix extension problem with symmetry
for the biothogonal case. We obtain a result on representing a pair
of $r\times s$ biorthogonal matrices $(\pP,\wt\pP)$ having the same
compatibly symmetry and provide a step-by-step algorithm for
deriving a pair of $s\times s$ biorthogonal matrices from a given
pair of biorthogonal matrices $(\pP,\wt\pP)$. Our results show that
for the one row case ($r=1$), the support lengths of the extension
matrices can be controlled by the given pair of columns.
We apply our results in this paper to the derivation of symmetric
biortahogonal multiwavelets from a pair of dual $\df$-refinable
functions.

\end{document}